\documentclass[11pt,reqno]{amsart}

\usepackage{amscd,amssymb,amsmath,amsthm}
\usepackage[arrow,matrix]{xy}
\usepackage{graphicx}
\usepackage{epstopdf}
\usepackage{color}
\newtheorem{thm}{Theorem}
\newtheorem{defn}{Definition}
\newtheorem{lemma}{Lemma}
\newtheorem{pro}{Proposition}
\newtheorem{rk}{Remark}
\newtheorem{cor}{Corollary}

\numberwithin{equation}{section} 

\def\r{\rho}

\def\Q{\mathbb Q}
\def\C{\mathbb C}

\def\C{\mathbb{C}}
\def\N{\mathbb N}

\begin{document}

\title[Dynamical systems of the $p$-adic $(2,2)$-rational functions]{Dynamical systems of the $p$-adic $(2,2)$-rational functions with two fixed
points}

\author{U.A. Rozikov, I.A. Sattarov}

 \address{U.\ A.\ Rozikov \\ Institute of mathematics,
81, Mirzo Ulug'bek str., 100170, Tashkent, Uzbekistan.} \email
{rozikovu@yandex.ru}

 \address{I.\ A.\ Sattarov \\ Institute of mathematics,
81, Mirzo Ulug'bek str., 100170, Tashkent, Uzbekistan.} \email
{sattarovi-a@yandex.ru}

\begin{abstract}

We consider a family of $(2,2)$-rational functions given on the set of complex
$p$-adic field $\C_p$. Each such function $f$ has the two distinct fixed points $x_1=x_1(f)$, $x_2=x_2(f)$. We study
$p$-adic dynamical systems generated by the $(2,2)$-rational functions.
We prove that $x_1$ is always indifferent fixed point for $f$,
i.e., $x_1$ is a center of some Siegel disk $SI(x_1)$.
Depending on the parameters of the function $f$,
the type of the fixed point $x_2$ may be any possibility: indifferent, attractor, repeller.
We find Siegel disk or basin of attraction of
the fixed point $x_2$, when $x_2$ is indifferent or attractor, respectively.
When $x_2$ is repeller we find an open ball any point of which repelled from $x_2$.
 Moreover, we study relations
between the sets $SI(x_1)$ and $SI(x_2)$ when $x_2$ is indifferent.
 For each $(2,2)$-rational function on $\C_p$ there are two points
 $\hat x_1=\hat x_1(f)$, $\hat x_2=\hat x_2(f)\in \C_p$ which are
 zeros of its denominator. We give explicit formulas of radiuses of spheres
(with the center at the fixed point $x_1$) containing some points such that
the trajectories (under actions of $f$) of the points after a finite step
come to $\hat x_1$ or $\hat x_2$.  We study periodic orbits of
 the dynamical system and find an invariant set, which contains all periodic orbits.
Moreover, we study ergodicity properties of the dynamical system on each invariant sphere.
Under some conditions we show that the system is ergodic iff $p=2$.
\end{abstract}

 \maketitle

\section{Introduction}

It is well-known that the completion of the set $\Q$ of rational numbers  with  respect to $p$-adic norm defines the
$p$-adic field which is denoted by $\Q_p$ (see \cite{Ko}).
The algebraic completion of $\Q_p$ is denoted by $\C_p$ and it is
called {\it the set of complex $p$-adic numbers}.

In this paper we consider $(2,2)$-rational function $f:\C_p\to\C_p$ defined by
\begin{equation}\label{fa}
f(x)=\frac{ax^2+bx+c}{x^2+dx+e}, \ \  x\neq \hat x_{1,2}=\frac{-d\pm\sqrt{d^2-4e}}{2}
\end{equation}
where parameters of the function satisfy the following conditions
\begin{equation}\label{par}
 a\neq 0,\ \ |b-ad|_p+|c-ae|_p\ne 0, \ \  a,b,c,d,e\in \C_p.
 \end{equation}

We study $p$-adic dynamical systems generated by the rational
function (\ref{fa}).
For basic definitions and motivations of such investigations see \cite{ARS}-\cite{S} and references therein.

The equation $f(x)=x$ for fixed points of function (\ref{fa}) is equivalent to  the equation
\begin{equation}\label{ce}
x^3+(d-a)x^2+(e-b)x-c=0.
\end{equation}
Since $\C_p$ is algebraic closed the equation (\ref{ce}) may have
three solutions with one of the following relations:

(i) One solution having multiplicity three;

(ii) Two solutions, one of which has multiplicity two;

(iii) Three distinct solutions.

\begin{rk} Since the behavior of the dynamical system depends on the set of fixed points,
each the above mentioned case (i)-(iii) has its own character of dynamics.  In recent paper \cite{RS2} the case (i) was considered.
In this paper we consider the case (ii), i.e., we investigate the behavior of trajectories of an
arbitrary $(2,2)$-rational dynamical system in complex $p$-adic
filed $\C_p$ when the there are two fixed points for $f$. The case (iii) will be considered in a separate
paper.
\end{rk}

The paper is organized as follows. In Section 2 we
show that (\ref{fa}) is conjugate to a simple (2,2)-rational function. Section 3 contains 8
real valued functions related to the $p$-adic $(2,2)$-rational functions.
These functions are very useful to study $p$-adic dynamical system by the real
functions. Section 4 devoted to full analysis of the behavior of the $p$-adic
dynamical system.  In the Section 5 we study periodic orbits and in the last
section we study ergodicity properties of the $p$-adic dynamical system on each invariant spheres.

\section{A function conjugate to $f$}

Denote by $x_1$ and $x_2$ solutions of equation (\ref{ce}), $x_2$ has multiplicity two.
Then we have $x^3+(d-a)x^2+(e-b)x-c=(x-x_1)(x-x_2)^2$ and

\begin{equation}\label{c3}
\left\{\begin{array}{lll}
x_1+2x_2=a-d \\[2mm]
x_2^2+2x_1x_2=e-b \\[2mm]
x_1x_2^2=c
\end{array}
\right.
\end{equation}

Let homeomorphism $h: \C_p \rightarrow \C_p$ is defined by $x=h(t)=t+x_2$. So $h^{-1}(x)=x-x_2$.
Note that, the function $f$ is topologically conjugate to function $h^{-1}\circ f\circ h$.
We have
\begin{equation}\label{fb}
(h^{-1}\circ f\circ h)(t)=\frac{At^2+Bt}{t^2+Dt+E},
\end{equation}
where $A=a-x_2$, $B=-2x_2^2+(2a-d)x_2+b$, $D=2x_2+d$ and $E=x_2^2+dx_2+e$.

It should be noted that $B=E\neq 0$ and $A\neq D$. Indeed, from (\ref{c3}) we have $$x_1=a-d-2x_2 \ \ {\rm and} \ \  x_2^2+2(a-d-2x_2)x_2=e-b.$$
Consequently, $-2x_2^2+(2a-d)x_2+b=x_2^2+dx_2+e$, i.e. $B=E$. From $x_2\neq\hat x_{1,2}$, we have $B=E\neq 0$.

Recall that, $f$ has two fixed points $x_1$, $x_2$ such that $x_1\neq x_2$. Assume that $A=D$. Then $x_2={{a-d}\over 3}$.
By (\ref{c3}) we have $x_1=x_2$. Therefore $x_1\neq x_2$ iff $A\neq D$.

Thus
$$(h^{-1}\circ f\circ h)(t)=\frac{At^2+Bt}{t^2+Dt+B},$$
and the following holds

\begin{pro}\label{pro1}
Any $(2,2)$-rational function having two distinct fixed points is topologically conjugate to a function of the following form
\begin{equation}\label{fd}
f(x)=\frac{ax^2+bx}{x^2+dx+b}, \ \ ab\neq 0,\ \ \ a\ne d, \ \  a,b,d\in \C_p.
\end{equation}
where $x\neq \hat x_{1,2}=\frac{-d\pm\sqrt{d^2-4b}}{2}$.
\end{pro}

Thus we study the dynamical system $(f,\C_p)$ with $f$ given by (\ref{fd}).

Note that, function (\ref{fd}) has two fixed points $x_1=0$ and $x_2=a-d$. So we have $$f'(x_1)=1 \ \ {\rm and} \ \ f'(x_2)={{b+(a-d)d}\over{b+(a-d)a}}.$$

Thus, the point $x_1$ is an indifferent point for (\ref{fd}).

For any $x\in \C_p$, $x\ne \hat x_{1,2}$, by simple calculations we
get
\begin{equation}\label{f2}
    |f(x)|_p=|x|_p\ \cdot{{|ax+b|_p}\over {|x-\hat x_1|_p|x-\hat x_2|_p}}.
\end{equation}

Denote
\begin{equation}\label{P}
\mathcal P=\{x\in \C_p: \exists n\in \N\cup\{0\}, f^n(x)\in\{\hat x_1, \hat x_2\}\}, \ \ \alpha=|\hat x_1|_p \ \ {\rm and} \ \ \beta=|\hat x_2|_p.
\end{equation}

Since $\hat x_1+\hat x_2=-d$ and $\hat x_1\hat x_2=b$, we have $|d|_p\leq\max\{\alpha, \beta\}$ and $|b|_p=\alpha\beta$.

\begin{rk}\label{r2} It is easy to see that $\hat x_1$ and $\hat x_2$
are symmetric in (\ref{f2}), i.e., if we replace them then RHS of
(\ref{f2}) does not change. Therefore we consider the dynamical
system $(f, \C_p\setminus\mathcal P)$ for cases $\alpha=\beta$ and $\alpha<\beta$.
\end{rk}

\section{Relation of the $p$-adic dynamics to several real ones}

By using (\ref{f2}) we define the following functions

1. For $|a|_p<\alpha=\beta$ define the function $\varphi_{\alpha}:
[0,+\infty)\to [0,+\infty)$ by
$$\varphi_{\alpha}(r)=\left\{\begin{array}{lllll}
r, \ \ {\rm if} \ \ r<\alpha\\[2mm]
\alpha^*, \ \ {\rm if} \ \ r=\alpha\\[2mm]
{{\alpha^2}\over r}, \ \ {\rm if} \ \ \alpha<r<{{\alpha^2}\over {|a|_p}}\\[2mm]
a^*, \ \ {\rm if} \ \ r={{\alpha^2}\over {|a|_p}}\\[2mm]
|a|_p, \ \ \ \ {\rm if} \ \ r>{{\alpha^2}\over {|a|_p}}
\end{array}
\right.
$$
where $\alpha^*$ and $a^*$ some positive numbers with
$\alpha^*\geq\alpha$, $a^*\leq |a|_p$.

2. For $|a|_p=\alpha=\beta$ define the function $\phi_{\alpha}:
[0,+\infty)\to [0,+\infty)$ by
$$\phi_{\alpha}(r)=\left\{\begin{array}{lll}
r, \ \ {\rm if} \ \ r<\alpha\\[2mm]
\hat\alpha, \ \ {\rm if} \ \ r=\alpha\\[2mm]
|a|_p, \ \  {\rm if} \ \ r>\alpha
\end{array}
\right.
$$
where  $\hat\alpha$ some positive number.

3. For $|a|_p>\alpha=\beta$ define the function $\psi_{\alpha}:
[0,+\infty)\to [0,+\infty)$ by
$$\psi_{\alpha}(r)=\left\{\begin{array}{lllll}
r, \ \ {\rm if} \ \ r<{{\alpha^2}\over {|a|_p}}\\[2mm]
a', \ \ {\rm if} \ \ r={{\alpha^2}\over {|a|_p}}\\[2mm]
{{{|a|_p} r^2}\over {\alpha^2}}, \ \ {\rm if} \ \ {{\alpha^2}\over {|a|_p}}<r<\alpha\\[2mm]
\alpha', \ \ {\rm if} \ \ r=\alpha\\[2mm]
|a|_p, \ \ \ \ {\rm if} \ \ r>\alpha
\end{array}
\right.
$$
where $a'$ and $\alpha'$ some positive numbers with
$a'\leq{{\alpha^2}\over {|a|_p}}$, $\alpha'\geq |a|_p$.

4. For $|a|_p<\alpha<\beta$ define the function $\varphi_{\alpha,\beta}:
[0,+\infty)\to [0,+\infty)$ by
$$\varphi_{\alpha,\beta}(r)=\left\{\begin{array}{lllllll}
r, \ \ {\rm if} \ \ r<\alpha\\[2mm]
\check\alpha, \ \ {\rm if} \ \ r=\alpha\\[2mm]
\alpha, \ \ {\rm if} \ \ \alpha<r<\beta\\[2mm]
\check\beta, \ \ {\rm if} \ \ r=\beta\\[2mm]
{{\alpha\beta}\over r}, \ \ {\rm if} \ \ \beta<r<{{\alpha\beta}\over {|a|_p}}\\[2mm]
\check a, \ \ {\rm if} \ \ r={{\alpha\beta}\over {|a|_p}}\\[2mm]
|a|_p, \ \ \ \ {\rm if} \ \ r>{{\alpha\beta}\over {|a|_p}}
\end{array}
\right.
$$
where $\check\alpha$, $\check\beta$ and $\check a$ some positive numbers with
$\check\alpha\geq\alpha$, $\check\beta\geq\alpha$ and $\check a\leq |a|_p$.

5. For $|a|_p=\alpha<\beta$ define the function $\phi_{\alpha,\beta}:
[0,+\infty)\to [0,+\infty)$ by
$$\phi_{\alpha,\beta}(r)=\left\{\begin{array}{lllll}
r, \ \ {\rm if} \ \ r<\alpha\\[2mm]
\tilde\alpha, \ \ {\rm if} \ \ r=\alpha\\[2mm]
\alpha, \ \ {\rm if} \ \ \alpha<r<\beta\\[2mm]
\tilde\beta, \ \ {\rm if} \ \ r=\beta\\[2mm]
|a|_p, \ \ \ \ {\rm if} \ \ r>\beta
\end{array}
\right.
$$
where $\tilde\alpha$ and $\tilde\beta$ some positive numbers with
$\tilde\alpha\geq\alpha$.

6. For $\alpha<|a|_p<\beta$ define the function $\psi_{\alpha,\beta}:
[0,+\infty)\to [0,+\infty)$ by
$$\psi_{\alpha,\beta}(r)=\left\{\begin{array}{lllllll}
r, \ \ {\rm if} \ \ r<\alpha\\[2mm]
\breve\alpha, \ \ {\rm if} \ \ r=\alpha\\[2mm]
\alpha, \ \ {\rm if} \ \ \alpha<r<{{\alpha\beta}\over {|a|_p}}\\[2mm]
\breve a, \ \ {\rm if} \ \ r={{\alpha\beta}\over {|a|_p}}\\[2mm]
{{|a|_pr}\over\beta}, \ \ {\rm if} \ \ {{\alpha\beta}\over {|a|_p}}<r<\beta\\[2mm]
\breve\beta, \ \ {\rm if} \ \ r=\beta\\[2mm]
|a|_p, \ \  {\rm if} \ \ r>\beta
\end{array}
\right.
$$
where  $\breve\alpha$, $\breve a$ and $\breve\beta$ some positive number with
$\breve\alpha\geq\alpha$, $\breve a\leq\alpha$ and $\breve\beta\geq |a|_p$.

7. For $\alpha<|a|_p=\beta$ define the function $\eta_{\alpha,\beta}:
[0,+\infty)\to [0,+\infty)$ by
$$\eta_{\alpha,\beta}(r)=\left\{\begin{array}{llll}
r, \ \ {\rm if} \ \ r<\beta, \, r\neq\alpha\\[2mm]
\acute\alpha, \ \ {\rm if} \ \ r=\alpha\\[2mm]
\acute\beta, \ \ {\rm if} \ \ r=\beta\\[2mm]
|a|_p, \ \ \ \ {\rm if} \ \ r>\beta
\end{array}
\right.
$$
where $\acute\alpha$ and $\acute\beta$ some positive numbers with
$\acute\beta\geq |a|_p$.

8. For $\alpha<\beta<|a|_p$ define the function $\zeta_{\alpha,\beta}:
[0,+\infty)\to [0,+\infty)$ by
$$\zeta_{\alpha,\beta}(r)=\left\{\begin{array}{lllllll}
r, \ \ {\rm if} \ \ r<{{\alpha\beta}\over {|a|_p}}\\[2mm]
\grave a, \ \ {\rm if} \ \ r={{\alpha\beta}\over {|a|_p}}\\[2mm]
{{{|a|_p} r^2}\over {\alpha\beta}}, \ \ {\rm if} \ \ {{\alpha\beta}\over {|a|_p}}<r<\alpha\\[2mm]
\grave\alpha, \ \ {\rm if} \ \ r=\alpha\\[2mm]
{{{|a|_p} r}\over {\beta}}, \ \ {\rm if} \ \ \alpha<r<\beta\\[2mm]
\grave\beta, \ \ {\rm if} \ \ r=\beta\\[2mm]
|a|_p, \ \ \ \ {\rm if} \ \ r>\beta
\end{array}
\right.
$$
where $\grave a$, $\grave\alpha$ and $\grave\beta$ some positive numbers with
$\grave a\leq{{\alpha\beta}\over {|a|_p}}$, $\grave\alpha\geq {{|a|_p\alpha}\over\beta}$ and $\grave\beta\geq |a|_p$.

Using formula (\ref{f2}) we easily get the following:

\begin{lemma}\label{lf2} If $x\in S_r(0)$, $x\neq\hat x_{1,2}$ then the following formula
holds for function (\ref{fd})
$$|f^n(x)|_p=\left\{\begin{array}{llllllll}
\varphi_{\alpha}^n(r), \ \ \mbox{if} \ \ |a|_p<\alpha=\beta\\[2mm]
\phi_{\alpha}^n(r), \ \ \mbox{if} \ \ |a|_p=\alpha=\beta\\[2mm]
\psi_{\alpha}^n(r), \ \ \mbox{if} \ \ |a|_p>\alpha=\beta\\[2mm]
\varphi_{\alpha,\beta}^n(r), \ \ \mbox{if} \ \ |a|_p<\alpha<\beta\\[2mm]
\phi_{\alpha,\beta}^n(r), \ \ \mbox{if} \ \ |a|_p=\alpha<\beta\\[2mm]
\psi_{\alpha,\beta}^n(r), \ \ \mbox{if} \ \ \alpha<|a|_p<\beta\\[2mm]
\eta_{\alpha,\beta}^n(r), \ \ \mbox{if} \ \ \alpha<|a|_p=\beta\\[2mm]
\zeta_{\alpha,\beta}^n(r), \ \ \mbox{if} \ \ \alpha<\beta<|a|_p.
\end{array}\right.$$
\end{lemma}
Thus the $p$-adic dynamical system $f^n(x), n\geq 1, x\in
\C_p\setminus{\mathcal P}$ is related to the real dynamical
systems generated by functions 1-8, and we have eight cases.

\section{Behavior of dynamical systems}

Note that $x_1$ is indifferent fixed point for $f$, i.e., $x_1$ is a center of some Siegel disk $SI(x_1)$.
In this section we define character of the fixed point $x_2$ for each cases.
Then we find Siegel disk or basin of attraction of
the fixed point $x_2$, when $x_2$ is indifferent or attractor, respectively.
In the case $x_2$ is repeller we find an open ball $U_r(x_2)$, such that the inequality
$|f(x)-x_2|_p>|x-x_2|_p$ holds for all $x\in U_r(x_2)$. Moreover, we study a relation
between the sets $SI(x_1)$ and $SI(x_2)$ when $x_2$ is indifferent.

\subsection{Case: $|a|_p<\alpha=\beta$.}

\begin{rk} In this case Theorem \ref{t1} gives the following character of the dynamical system: the
open ball with radius $\alpha$ and center $x_1$ is the maximal Siegel disk for fixed point $x_1$. The fixed point $x_2$
may be attractor or indifferent. If $x_2$ is an attractor then its basin of attraction is $A(x_2)=U_{\alpha}(x_2)\subset S_{\alpha}(0).$
If fixed point $x_2$ is indifferent then it have the Siegel disk and

$1. \, SI(x_2)=SI(x_1), \, if \, |d|_p<\alpha;$

$2. \, SI(x_2)\cap SI(x_1)=\emptyset, \, if \, |d|_p=\alpha.$
\end{rk}

\begin{lemma}\label{l1} If $|a|_p<\alpha=\beta$, then the dynamical
system generated by $\varphi_{\alpha}(r)$ has
the following properties:
\begin{itemize}
\item[1.] ${\rm Fix}(\varphi_{\alpha})=\{r: 0\leq r<\alpha\}\cup\{\alpha:\, if \,
\alpha^*=\alpha\}$.
\item[2.] If $r>\alpha$, then
$$\varphi_{\alpha,\delta}^n(r)=\left\{\begin{array}{lll}
{{\alpha^2}\over r}, \ \ \mbox{for all} \ \ \alpha<r<{{\alpha^2}\over {|a|_p}}\\[2mm]
a^*, \ \ \mbox{for} \ \ r={{\alpha^2}\over {|a|_p}}\\[2mm]
|a|_p, \ \ \ \ \mbox{for all} \ \ r>{{\alpha^2}\over {|a|_p}}
\end{array}
\right.$$ $\mbox{for any} \ \ n\geq 1$.
\item[3.] If $r=\alpha$ and $\alpha^*>\alpha$, then
$$\varphi_{\alpha}^n(r)=\left\{\begin{array}{lll}
{{\alpha^2}\over \alpha^*}, \ \ \mbox{if} \ \ \alpha<\alpha^*<{{\alpha^2}\over {|a|_p}}\\[2mm]
a^*, \ \ \mbox{if} \ \ \alpha^*={{\alpha^2}\over {|a|_p}}\\[2mm]
|a|_p, \ \ \ \ \mbox{if} \ \ \alpha^*>{{\alpha^2}\over {|a|_p}}
\end{array}
\right.$$ $\mbox{for any} \ \ n\geq 2$.
\end{itemize}
\end{lemma}

\begin{proof} 1. This is the result of a simple analysis
of the equation $\varphi_{\alpha}(r)=r$.

2. If $r>\alpha$, then
$$\varphi_{\alpha}(r)=\left\{\begin{array}{lll}
{\alpha^2\over r}, \ \ {\rm if} \ \ \alpha<r<{{\alpha^2}\over {|a|_p}}\\[2mm]
a^*, \ \ {\rm if} \ \ r={{\alpha^2}\over {|a|_p}}\\[2mm]
|a|_p, \ \ {\rm if} \ \ r>{{\alpha^2}\over {|a|_p}}.
\end{array}
\right.
$$
Consequently,
$$\alpha<r<{{\alpha^2}\over {|a|_p}} \ \ \Rightarrow \ \ {|a|_p}<{{\alpha^2}\over r}<\alpha \ \ \Rightarrow \ \ \varphi_{\alpha}(r)<\alpha.$$
If $r\geq{{\alpha^2}\over {|a|_p}}$, then by
$a^*\leq{|a|_p}<\alpha$ we have
$\varphi_{\alpha}(r)<\alpha$.
Thus $\varphi_{\alpha}(\varphi_{\alpha}(r))=\varphi_{\alpha}(r)$,
i.e., $\varphi_{\alpha}(r)$ is a fixed point of
$\varphi_{\alpha}$ for any $r>\alpha$. Consequently, for each $n\geq 1$ we have
$$\varphi_{\alpha}^n(r)=\left\{\begin{array}{lll}
{\alpha^2\over r}, \ \ {\rm if} \ \ \alpha<r<{{\alpha^2}\over{|a|_p}}\\[2mm]
a^*, \ \ {\rm if} \ \ r={{\alpha^2}\over{|a|_p}}\\[2mm]
{|a|_p}, \ \ {\rm if} \ \ r>{{\alpha^2}\over{|a|_p}}.
\end{array}
\right.$$

3. The part 3 easily follows from the parts 1 and 2.
\end{proof}

Now we shall apply these lemmas to study the $p$-adic
dynamical system generated by function (\ref{fd}).

For $|a|_p<\alpha=\beta$ denote the following
$$a^*(x)=|f(x)|_p, \ \ {\rm if} \ \ x\in
S_{{\alpha^2}\over {|a|_p}}(0).$$

Then using Lemma \ref{lf2} and
Lemma \ref{l1} we obtain the following

 \begin{thm}\label{t1} If $|a|_p<\alpha=\beta$, then
 the $p$-adic dynamical system generated by
 function (\ref{fd}) has the following properties:
\begin{itemize}
\item[1.]
\begin{itemize}
\item[1.1)] $SI(x_1)=U_{\alpha}(0)$.
\item[1.2)] $\mathcal P\subset S_{\alpha}(0)$.
\end{itemize}
\item[2.] If $r>\alpha$ and $x\in S_{r}(0)$,
then $$f^n(x)\in\left\{\begin{array}{lll}
S_{\alpha^2\over r}(0), \ \ \mbox{for all} \ \ \alpha<r<{{\alpha^2}\over {|a|_p}}\\[2mm]
S_{a^*(x)}(0), \ \ \mbox{for} \ \ r={{\alpha^2}\over {|a|_p}}\\[2mm]
S_{|a|_p}(0), \ \ \mbox{for all} \ \ r>{{\alpha^2}\over
{|a|_p}},
\end{array}
\right.$$ for any $n\geq 1$.
\item[3.] If $x\in S_{\alpha}(0)\setminus\mathcal P$,
then one of the following two possibilities holds:
\begin{itemize}
\item[3.1)]There exists $k\in N$ and $\mu_k>\alpha$
such that $f^k(x)\in S_{\mu_k}(0)$ and $$f^m(x)\in\left\{\begin{array}{lll}
S_{\alpha^2\over {\mu_k}}(0), \ \ \mbox{if} \ \ \alpha<\mu_k<{{\alpha^2}\over {|a|_p}}\\[2mm]
S_{a^*(f^k(x))}(0), \ \ \mbox{if} \ \ \mu_k={{\alpha^2}\over {|a|_p}}\\[2mm]
S_{{|a|_p}}(0), \ \ \mbox{if} \ \ \mu_k>{{\alpha^2}\over
{|a|_p}}
\end{array}
\right.$$ for any $m\geq k+1$ and $f^m(x)\in S_{\alpha}(0)$ if
$m\leq k-1$.
\item[3.2)] The trajectory $\{f^k(x), k\geq 1\}$ is a subset of
$S_{\alpha}(0)$.
\end{itemize}
\item[4.] If $|d|_p<\alpha$, then $|f'(x_2)|_p=1$ and $$SI(x_2)=SI(x_1).$$
\item[5.] Let $|d|_p=\alpha$. Then  $x_2\in S_{\alpha}(0)$ and
\begin{itemize}
\item[5.1)] if $|b+(a-d)d|_p<\alpha^2$, then $x_2$ is an attractive fixed point for $f$ and its basin of attraction is $$A(x_2)=U_{\alpha}(x_2)\subset S_{\alpha}(0).$$
\item[5.2)] if $|b+(a-d)d|_p=\alpha^2$, then $x_2$ is indifferent fixed point for $f$ and $$SI(x_2)=U_{\alpha}(x_2)\subset S_{\alpha}(0).$$
\end{itemize}
\end{itemize}
\end{thm}
\begin{proof} We shall prove the part 1, by using parts 2 and 3.

The part 2 easily follows from Lemma \ref{lf2} and  the part 2 of Lemma \ref{l1}.

3. Take $x\in S_\alpha(0)\setminus \mathcal P$ then by (\ref{f2}) we have
$$
|f(x)|_p={{\alpha^3}\over{\left|(x-\hat x_1)(x-\hat x_2)\right|_p}}\geq\alpha.
$$
If $|f(x)|_p>\alpha$ then there is $\mu_1>\alpha$ such that
$f(x)\in S_{\mu_1}(0)$ and by part 2 we have
$$f^m(x)\in\left\{\begin{array}{lll}
S_{\alpha^2\over {\mu_1}}(0), \ \ \mbox{if} \ \ \alpha<\mu_1<{{\alpha^2}\over {|a|_p}}\\[2mm]
S_{a^*(f(x))}(0), \ \ \mbox{if} \ \ \mu_1={{\alpha^2}\over {|a|_p}}\\[2mm]
S_{{|a|_p}}(0), \ \ \mbox{if} \ \ \mu_1>{{\alpha^2}\over
{|a|_p}}
\end{array}
\right.$$ for any $m\geq 2$. So in this case $k=1$.

If $|f(x)|_p=\alpha$ then we consider the following
$$
|f^2(x)|_p={{\alpha^3}\over{\left|(f(x)-\hat x_1)(f(x)-\hat x_2)\right|_p}}\geq\alpha.
$$
Now, if $|f^2(x)|_p>\alpha$ then there is $\mu_2>\alpha$ such
that $f^2(x)\in S_{\mu_2}(0)$ and by part 2 we get
$$f^m(x)\in\left\{\begin{array}{lll}
S_{\alpha^2\over {\mu_2}}(0), \ \ \mbox{if} \ \ \alpha<\mu_2<{{\alpha^2}\over {|a|_p}}\\[2mm]
S_{a^*(f^2(x))}(0), \ \ \mbox{if} \ \ \mu_2={{\alpha^2}\over {|a|_p}}\\[2mm]
S_{{|a|_p}}(0), \ \ \mbox{if} \ \ \mu_2>{{\alpha^2}\over
{|a|_p}}
\end{array}
\right.$$ for any $m\geq 3$. So in this case $k=2$.

If $|f^2(x)|_p=\alpha$ then we can continue the argument and
get the following inequality
$$|f^k(x)|_p\geq\alpha.$$
Hence in each step we may have two possibilities:
$|f^k(x)|_p=\alpha$ or $|f^k(x)|_p>\alpha$. In case
$|f^k(x)|_p>\alpha$ there exists $\mu_k$ such that $f^k(x)\in
S_{\mu_k}(0)$, and
$$f^m(x)\in\left\{\begin{array}{lll}
S_{\alpha^2\over {\mu_k}}(0), \ \ \mbox{if} \ \ \alpha<\mu_k<{{\alpha^2}\over {|a|_p}}\\[2mm]
S_{a^*(f^k(x))}(0), \ \ \mbox{if} \ \ \mu_k={{\alpha^2}\over {|a|_p}}\\[2mm]
S_{{|a|_p}}(0), \ \ \mbox{if} \ \ \mu_k>{{\alpha^2}\over
{|a|_p}}
\end{array}
\right.$$ for any $m\geq k+1$. If $|f^k(x)|_p=\alpha$ for any
$k\in \N$ then $\{f^k(x), k\geq 1\}\subset S_\alpha(0)$.

1. By parts 2 and 3 of theorem we know that $S_r(0)$ is not an invariant of $f$ for
$r\geq\alpha$. Consequently, $SI(x_1)\subset U_{\alpha}(0)$.

By Lemma \ref{lf2} and part 1 of  Lemma \ref{l1} if
$r<\alpha$ and $x\in S_r(0)$ then
$|f^n(x)|_p=\varphi^n_{\alpha}(r)=r$, i.e., $f^n(x)\in
S_r(0)$. Hence $U_{\alpha}(0)\subset SI(x_1)$ and thus
$SI(x_1)=U_{\alpha}(0).$

Since $|\hat x_1|_p=|\hat x_2|_p=\alpha$ we have $\hat x_i\not\in
U_{\alpha}(0), \, i=1,2$. From $f(U_{\alpha}(0))\subset
U_{\alpha}(0)$ it follows that $$U_{\alpha}(0)\cap\mathcal P=\{x\in
U_{\alpha}(0): \exists n\in N\cup\{0\}, \, f^n(x)\in\{\hat x_1, \hat x_2\}
\}=\emptyset.$$
By part 2 of theorem for $r>\alpha$ we have
$f(S_r(0))\subset U_{\alpha}(0)$. Let $V_{\alpha}(0)$ be closed ball with the center $0$ and radius $\alpha$. Then
$$(\C_p\setminus{V_{\alpha}(0)})\cap\mathcal P=\emptyset,$$
i.e., $\mathcal P\subset S_{\alpha}(0)$.

4. Note that $|a|_p<\alpha$ and $|d|_p\leq\alpha$. If $|d|_p<\alpha$, then $|x_2|_p=|a-d|_p<\alpha$.
So $x_2\in U_{\alpha}(0)=SI(x_1)$ and $$|f'(x_2)|_p={{|b+(a-d)d|_p}\over{|b+(a-d)a|_p}}={{|b|_p}\over{|b|_p}}=1.$$ Consequently
$x_2$ is indifferent fixed point for $f$ and
\begin{equation}\label{si}
SI(x_2)\subset SI(x_1).
\end{equation}
By simple calculation we get
\begin{equation}\label{fx2}
|f(x)-x_2|_p=|x-x_2|_p\cdot{{|d(x-x_2)+dx_2+b|_p}\over{|(x-x_2)+(x_2-\hat x_1)|_p|(x-x_2)+(x_2-\hat x_2)|_p}}.
\end{equation}
If $x\in S_{\rho}(x_2)\subset U_{\alpha}(0)$, for some $\rho< \alpha$, then in (\ref{fx2}) we have $|d(x-x_2)+dx_2+b|_p=\alpha^2$. Moreover, $|x_2-\hat x_1|_p=|a+\hat x_2|_p=\alpha$ and $|x_2-\hat x_2|_p=|a+\hat x_1|_p=\alpha$.
Therefore, $|f(x)-x_2|_p=|x-x_2|_p$, i.e. $f(x)\in S_{\rho}(x_2)$ holds for every $x\in S_{\rho}(x_2)\subset U_{\alpha}(x_2)$.
Then $U_{\alpha}(x_2)=U_{\alpha}(0)=SI(x_1)\subset SI(x_2)$ and by (\ref{si}) we have $SI(x_2)=SI(x_1)$.

5. If $|d|_p=\alpha$, then $|x_2|_p=|a-d|_p=|d|_p=\alpha$, i.e., $x_2\in S_{\alpha}(0)$. Moreover, if $x\in U_{\alpha}(x_2)$, then  $|x|_p=|(x-x_2)+x_2|_p=\alpha$, i.e., $U_{\alpha}(x_2)\subset S_{\alpha}(0)$.

 Note that $$|f'(x_2)|_p={{|b+(a-d)d|_p}\over{|b+(a-d)a|_p}}.$$ We have $|b+(a-d)d|_p\leq\alpha^2$ and $|b+(a-d)a|_p=\alpha^2$.

5.1. If $|b+(a-d)d|_p<\alpha^2$, then $|f'(x_2)|_p< 1$, i.e., $x_2$ is attractive fixed point for $f$. If $x\in U_{\alpha}(x_2)$,
 then in (\ref{fx2}) we have $|d(x-x_2)+dx_2+b|_p<\alpha^2$. Therefore, $|f(x)-x_2|_p<|x-x_2|_p$ for all $x\in U_{\alpha}(x_2)$.
 So $$U_{\alpha}(x_2)\subset A(x_2).$$
If $x\in S_{\alpha}(0)\setminus(U_{\alpha}(x_2)\cup\mathcal P)$, then $|x-x_2|_p=\alpha$ and by (\ref{fx2}) we have $|f(x)-x_2|_p\geq|x-x_2|_p$, i.e., $x\not\in A(x_2)$. Consequently, $$A(x_2)=U_{\alpha}(x_2).$$

5.2. If $|b+(a-d)d|_p=\alpha^2$, then $|f'(x_2)|_p= 1$, i.e., $x_2$ is indifferent fixed point for $f$. If $x\in S_{\rho}(x_2)\subset U_{\alpha}(x_2)$, then in (\ref{fx2}) we have $|d(x-x_2)+dx_2+b|_p=\alpha^2$ and $|f(x)-x_2|_p=|x-x_2|_p$. Therefore, $f(x)\in S_{\rho}(x_2)$ for all $x\in S_{\rho}(x_2)$. So $U_{\alpha}(x_2)\subset SI(x_2).$

If $x\in S_{\alpha}(0)\setminus(U_{\alpha}(x_2)\cup\mathcal P)$, then $|x-x_2|_p=\alpha$ and by (\ref{fx2}) we have $|f(x)-x_2|_p$ is some given number with $|f(x)-x_2|_p>0$, i.e., the sphere $S_{\alpha}(x_2)$ is not invariant for $f$. Consequently, $$SI(x_2)=U_{\alpha}(x_2).$$
\end{proof}

\subsection{Case: $|a|_p=\alpha=\beta$.}

\begin{rk} In this case Theorem \ref{t2} gives the following character of the dynamical system:
the open ball with radius $\alpha$ and center $x_1$ is the maximal Siegel disk for fixed point $x_1$.
The fixed point $x_2$ may be repeller or indifferent. If $x_2$ is repeller then
the inequality $|f(x)-x_2|_p>|x-x_2|_p$ holds for all $x\in U_{\alpha}(x_2), \, x\neq x_2$.
If $x_2$ is indifferent then it have the Siegel disk and

$1. \, SI(x_2)=SI(x_1), \, if \, |x_2|_p<\alpha;$

$2. \, SI(x_2)\cap SI(x_1)=\emptyset, \, if \, |x_2|_p=\alpha.$
\end{rk}

\begin{lemma}\label{l2} If $|a|_p=\alpha=\beta$, then the dynamical
system generated by $\phi_{\alpha}(r)$ has the following properties:
 \begin{itemize}
\item[1.] ${\rm Fix}(\phi_{\alpha})=
\{r: 0\leq r<\alpha\}\cup\{\alpha: \, if \,
\hat\alpha=\alpha\}$.
\item[2.] If $r>\alpha$, then
$\phi_{\alpha}(r)=\alpha.$
\item[3.] Let $r=\alpha$.
\begin{itemize}
\item[3.1)] If $\hat\alpha>\alpha$, then $\phi_{\alpha}^2(\alpha)=\alpha$.
\item[3.2)] If $\hat\alpha\leq\alpha$, then $\phi_{\alpha}^n(\alpha)=\hat\alpha$ for any $n\geq 1$.
\end{itemize}
\end{itemize}
\end{lemma}

\begin{proof} 1. This is the result of a simple
analysis of the equation $\phi_{\alpha}(r)=r$.

2. By definition of $\phi_{\alpha}(r)$, for any
$r>\alpha$ we have $\phi_{\alpha}(r)=\alpha$.

3. If $r=\alpha$ then $\phi_{\alpha}(r)=\hat\alpha$.

For $\hat\alpha>\alpha$ we have
$\phi_{\alpha}(\hat\alpha)=\alpha$,
$\phi_{\alpha}(\alpha)=\hat\alpha$. Hence
$\phi^2_{\alpha}(\alpha)=\alpha.$

In case $\hat\alpha\leq\alpha$ we have
$\phi_{\alpha}(\hat\alpha)=\hat\alpha$. Thus for all
$n\geq 1$ one has $\phi_{\alpha}^{n}(r)=\hat\alpha$.
\end{proof}

Denote
$$\hat\alpha(x)=|f(x)|_p, \
\ {\rm for} \ \ x\in S_{\alpha}(0).$$

Using Lemma \ref{lf2} and Lemma \ref{l2} we get

\begin{thm}\label{t2} If $|a|_p=\alpha=\beta$, then
 the $p$-adic dynamical system generated by function (\ref{fd}) has the following properties:
\begin{itemize}
\item[1.]
\begin{itemize}
\item[1.1)] $SI(x_1)=U_{\alpha}(0)$.
\item[1.2)] $U_{\alpha}(0)\cap{\mathcal P}=\emptyset$.
\end{itemize}
\item[2.] If $r>\alpha$ and $x\in S_r(0)$, then $f(x)\in
S_{\alpha}(0)$.
\item[3.] Let $f^k(x)\in S_{\alpha}(0)\setminus\mathcal P$ for some $k=0,1,2,...$,
then
$$f^m(x)\in\left\{\begin{array}{lll}
S_{\hat\alpha(f^k(x))}(0), \ \ \mbox{if} \ \ \hat\alpha(f^k(x))\geq\alpha, \ \ m=k+1\\[2mm]
S_{\alpha}(0), \ \ \ \ \ \ \ \ \ \mbox{if} \ \ \hat\alpha(f^k(x))>\alpha, \ \ m=k+2\\[2mm]
S_{\hat\alpha(f^k(x))}(0), \ \ \mbox{if} \ \ \hat\alpha(f^k(x))<\alpha, \ \ \forall m\geq k+1.
\end{array}
\right.$$
\item[4.] If $|d|_p<\alpha$, then  $|x_2|_p=\alpha$ and
\begin{itemize}
\item[4.1)] if $|b+(a-d)a|_p<\alpha^2$, then $x_2$ is repeller fixed point for $f$ and the inequality $|f(x)-x_2|_p>|x-x_2|_p$ holds for all $x\in U_{\alpha}(x_2), \, x\neq x_2$.
\item[4.2)] if $|b+(a-d)a|_p=\alpha^2$, then $x_2$ is indifferent fixed point for $f$ and $$SI(x_2)=U_{\alpha}(x_2).$$
\end{itemize}
\item[5.] If $|x_2|_p<\alpha$, then $x_2$ is indifferent fixed point for $f$ and $SI(x_2)=SI(x_1)$.
\end{itemize}
\end{thm}
\begin{proof} Parts 2-3 of theorem easily follow from parts 2-3 of Lemma \ref{l2}
and Lemma \ref{lf2}.

1.  Note that $x_1=0$ is indifferent fixed point for $f$. By the part 1 of Lemma \ref{l2} and Lemma \ref{lf2},
if $r<\alpha$ and $x\in S_r(0)$ then
$$|f^n(x)|_p=\phi^n(r)=r,$$ i.e., for  $n\geq 1$ we have $f^n(x)\in
S_r(0)$. Consequently, $U_{\alpha}(0)\subset SI(x_1)$.

By the parts 2-3 of theorem we know that if $r\geq\alpha$
then $S_r(0)$ is not invariant for $f$. Hence $SI(x_1)\subset U_{\alpha}(0)$.
Therefore, $SI(x_1)= U_{\alpha}(0)$.

Since $|\hat x_1|_p=|\hat x_2|_p=\alpha$ we have $\hat x_{1,2}\not\in
U_{\alpha}(0)$. Moreover, from $f(U_{\alpha}(0))\subset U_{\alpha}(0)$
we get
$$U_{\alpha}(0)\cap\mathcal P=\{x\in U_{\alpha}(0):
\exists n\in \N\cup\{0\}, f^n(x)\in\{\hat x_1, \hat x_2\} \}=\emptyset.$$

4. If $|d|_p<\alpha$, then $|x_2|_p=|a-d|_p=|a|_p=\alpha$, i.e., $x_2\in S_{\alpha}(0)$. Moreover, $U_{\alpha}(x_2)\subset S_{\alpha}(0)$.

 Note that $$|f'(x_2)|_p={{|b+(a-d)d|_p}\over{|b+(a-d)a|_p}}.$$ We have $|b+(a-d)d|_p=\alpha^2$ and $|b+(a-d)a|_p\leq\alpha^2$.

4.1. If $|b+(a-d)a|_p<\alpha^2$, then $|f'(x_2)|_p> 1$, i.e., $x_2$ is repeller fixed point for $f$.

By simple calculation we get
\begin{equation}\label{fxx}
|f(x)-x_2|_p={{|x-x_2|_p|d(x-x_2)+dx_2+b|_p}\over{|(x-x_2)^2+(x-x_2)(2x_2+d)+b+(a-d)a|_p}}.
\end{equation}

If $|b+(a-d)a|_p<\alpha^2$ and $x\in U_{\alpha}(x_2)$,
 then in (\ref{fxx}) we have $$|d(x-x_2)+dx_2+b|_p=\alpha^2 \ \ \mbox{and} \ \ |(x-x_2)^2+(x-x_2)(2x_2+d)+b+(a-d)a|_p<\alpha^2.$$ Therefore, the inequality $|f(x)-x_2|_p>|x-x_2|_p$ satisfied for all $x\in U_{\alpha}(x_2)$ $x\neq x_2$.

4.2. If $|b+(a-d)a|_p=\alpha^2$, then $|f'(x_2)|_p= 1$, i.e., $x_2$ is indifferent fixed point for $f$. If $x\in S_{\rho}(x_2)\subset U_{\alpha}(x_2)$, then in (\ref{fxx}) we have
$$|d(x-x_2)+dx_2+b|_p=\alpha^2,$$
$$|(x-x_2)^2+(x-x_2)(2x_2+d)+b+(a-d)a|_p=\alpha^2.$$ Therefore, $|f(x)-x_2|_p=|x-x_2|_p$, i.e.,  $f(x)\in S_{\rho}(x_2)$ for all $x\in S_{\rho}(x_2)$. So $U_{\alpha}(x_2)\subset SI(x_2).$

If $x\in S_{\alpha}(0)\setminus(U_{\alpha}(x_2)\cup\mathcal P)$, then $|x-x_2|_p=\alpha$ and by (\ref{fxx}) we have $|f(x)-x_2|_p$ is some given number with $|f(x)-x_2|_p\geq\alpha$, i.e., the sphere $S_{\alpha}(x_2)$ is not invariant for $f$. Consequently, $$SI(x_2)=U_{\alpha}(x_2).$$

5. If $|x_2|_p<\alpha$, then $x_2\in U_{\alpha}(0)$ and
$$|f'(x_2)|_p={{|b+dx_2|_p}\over{|b+ax_2|_p}}=1.$$

Consequently, $x_2$ is indifferent fixed point for $f$ and $SI(x_2)\subset U_{\alpha}(x_2)=U_{\alpha}(0)$,
because by parts 2-3 of theorem $f(S_r(0))\nsubseteq S_r(0)$ for all $r\geq \alpha$.
Note that $|x_2-\hat x_1|_p=|x_2-\hat x_2|_p=\alpha$.

Let $x\in S_{\rho}(x_2)$ and $\rho<\alpha$. By (\ref{fxx}) we have
$|f(x)-x_2|_p=|x-x_2|_p$, i.e. $f(x)\in S_{\rho}(x_2)$ holds for every $x\in S_{\rho}(x_2)\subset U_{\alpha}(x_2)$
Then $U_{\alpha}(x_2)=U_{\alpha}(0)\subset SI(x_2)$ and we have $$SI(x_2)=SI(x_1).$$
\end{proof}

\subsection{Case: $|a|_p>\alpha=\beta$.}

\begin{rk} In this case the open ball with radius ${{\alpha^2}\over{|a|_p}}$ and center $x_1$ is the maximal Siegel disk for fixed point $x_1$. The fixed point $x_2$ is attractive and its basin of attraction is the set $$A(x_2)=\C_p\setminus(V_{{\alpha^2}\over{|a|_p}}(0)\cup\mathcal P).$$ Moreover, Theorem \ref{tpk} gives explicit formulas of radiuses of spheres (with the center at the fixed point $x_1$) containing some points that the trajectories (under actions of $f$) of the points after a finite step come to $\hat x_1$ or $\hat x_2$.
\end{rk}

\begin{lemma}\label{l3} If $|a|_p>\alpha=\beta$, then the dynamical
system generated by $\psi_{\alpha}(r)$ has the following properties:
 \begin{itemize}
\item[1.] ${\rm Fix}(\psi_{\alpha})=
\{r: 0\leq r<{{\alpha^2}\over |a|_p}\}\cup\{{{\alpha^2}\over |a|_p}:\, if \,
\alpha'={{\alpha^2}\over |a|_p}\}\cup\{|a|_p\}$.
\item[2.] If $r>{{\alpha^2}\over |a|_p}$, then
$$\lim_{n \to \infty}\psi_{\alpha}^n(r)=|a|_p.$$
\item[3.] If $r={{\alpha^2}\over |a|_p}$ and $\alpha'<{{\alpha^2}\over |a|_p}$, then
$\psi_{\alpha}^n(r)=\alpha'$ $\mbox{for all} \ \ n\geq 1$.
\end{itemize}
\end{lemma}

\begin{proof} 1. This is the result of a simple analysis of the
equation $\psi_{\alpha}(r)=r$.

2. By definition of
$\psi_{\alpha}(r)$, for $r>\alpha$ we have
$\psi_{\alpha}(r)=|a|_p$, i.e., the function is constant.
For $r=\alpha$ we have $\psi_{\alpha}(\alpha)=a'\geq |a|_p$ and by condition
$|a|_p>\alpha$, we get
$\psi_{\alpha}(\alpha)>\alpha$. Consequently,
$$\lim_{n \to \infty}\psi_{\alpha}^n(\alpha)=|a|_p.$$
Assume now  ${{\alpha^2}\over |a|_p}<r<\alpha$ then
$\psi_{\alpha}(r)={{|a|_p r^2}\over {\alpha^2}}$,
$\psi'_{\alpha}(r)={{2|a|_p r}\over {\alpha^2}}>2$ and
$$\psi_{\alpha}(({{\alpha^2}\over
|a|_p},\alpha))=({{\alpha^2}\over
|a|_p},|a|_p)\cup\{a'\}.$$
Since  $\psi'_{\alpha}(r)>2$
for $r\in ({{\alpha^2}\over |a|_p},\alpha)$ there exists $n_0\in N$ such that
$\psi_{\alpha}^{n_0}(r)\in (\alpha,|a|_p)$.
Hence for $n\geq
n_0$ we get $\psi_{\alpha}^n(r)>\alpha$ and consequently
$$\lim_{n \to \infty}\psi_{\alpha}^n(r)=|a|_p.$$

3. If $r={{\alpha^2}\over |a|_p}$ and $\alpha'<{{\alpha^2}\over
|a|_p}$ then
$\psi_{\alpha}(r)=\alpha'<{{\alpha^2}\over |a|_p}$. Moreover,
$\alpha'$ is a fixed point for the function $\psi_{\alpha}$.
Thus for $n\geq 1$ we obtain
$\psi_{\alpha}^n(r)=\alpha'.$
\end{proof}

By Lemma \ref{lf2} and Lemma
\ref{l3} we get

\begin{thm}\label{t3} If $|a|_p>\alpha=\beta$, then
 the $p$-adic dynamical system generated by function (\ref{fd}) has the following properties:
\begin{itemize}
\item[1.]
\begin{itemize}
\item[1.1)] $SI(x_1)=U_{{{\alpha^2}\over |a|_p}}(0)$.
\item[1.2)] $x_2\in S_{|a|_p}(0)$.
\end{itemize}
\item[2.)] The fixed point $x_2$ is attractive and $$A(x_2)=C_p\setminus(V_{{{\alpha^2}\over |a|_p}}(0)\cup\mathcal P).$$
\item[3.] If $x\in S_{{\alpha^2}\over{|a|_p}}(0)$, then one of the
following two possibilities holds:
\begin{itemize}
\item[3.1)] There exists $k\in N$ and
$\mu_k<{{\alpha^2}\over {|a|_p}}$ such that $f^m(x)\in
S_{\mu_k}(0)$ for any $m\geq k$ and $f^m(x)\in
S_{{\alpha^2}\over{|a|_p}}(0)$ if $m\leq k-1$.
\item[3.2)] The trajectory $\{f^k(x), k\geq 1\}$ is a subset of
$S_{{\alpha^2}\over{|a|_p}}(0)$.
\end{itemize}
\end{itemize}
\end{thm}
\begin{proof}
1. By Lemma \ref{lf2} and part 1 of Lemma \ref{l3} we see that
 spheres $S_r(0)$, $r<{{\alpha^2}\over{|a|_p}}$ and $S_{|a|_p}(0)$ are invariant for $f$.
 Thus $SI(x_1)=U_{{\alpha^2}\over{|a|_p}}(0)$.

Note that $|a|_p>\alpha$ and $|d|_p\leq\alpha$. Consequently, $|x_2|_p=|a-d|_p=|a|_p$, i.e. $x_2\in S_{|a|_p}(0)$.

2. In this case $x_2$ will be attractive fixed point, i.e. $$|f'(x_2)|_p={{|b+dx_2|_p}\over{|b+ax_2|_p}}<1,$$
because we have $|b+dx_2|_p\leq\max\{\alpha^2, |ad|_p\}<|a|^2_p$ and
$|b+ax_2|_p=|a|^2_p$.

From Lemma \ref{lf2} and part 2 of Lemma
\ref{l3} we have $$\lim_{n\to\infty}f^n(x)\in
S_{|a|_p}(0)$$ for all $x\in S_r(0)\setminus\mathcal P$, $r>{{\alpha^2}\over{|a|_p}}$.

Let $x\in S_{|a|_p}(0)$. We have
\begin{equation}\label{fx3}
|f(x)-x_2|_p=|x-x_2|_p\cdot{{|dx+b|_p}\over{|(x-\hat x_1)|_p|(x-\hat x_2)|_p}}.
\end{equation}
By $|dx+b|_p\leq\max\{|ad|_p, \alpha^2\}<|a|^2_p$ and $|x-\hat x_1|_p=|x-\hat x_2|_p=|a|_p$,
we get $|f(x)-x_2|_p<|x-x_2|_p$ for any $x\in S_{|a|_p}(0)\setminus\mathcal P$.
Consequently, $$\lim_{n\to\infty}f^n(x)=x_2,  \ \ \mbox{for all} \ \ x\in S_r(0)\setminus\mathcal P, \ \ r>{{\alpha^2}\over{|a|_p}},$$
i.e., $A(x_2)=C_p\setminus(V_{{\alpha^2}\over{|a|_p}}(0)\cup\mathcal P)$.

3. If $x\in S_{{\alpha^2}\over{|a|_p}}(0)$ then by (\ref{f2}) we have
$$
|f(x)|_p={|ax+b|_p\over{|a|_p}}\leq{{\alpha^2}\over{|a|_p}}.
$$
If $|f(x)|_p<{{\alpha^2}\over{|a|_p}}$ then there is
$\mu_1<{{\alpha^2}\over{|a|_p}}$ such that $f^m(x)\in
S_{\mu_1}(0)$ for any $m\geq 1$ (see part 1 of Lemma \ref{l3}). So in this case $k=1$.

If $|f(x)|_p={{\alpha^2}\over{|a|_p}}$ then we consider the
following
$$
|f^2(x)|_p={|af(x)+b|_p\over{{|a|_p}}}\leq{{\alpha^2}\over{|a|_p}}.
$$
Now, if $|f^2(x)|_p<{{\alpha^2}\over{|a|_p}}$ then there is
$\mu_2<{{\alpha^2}\over{|a|_p}}$ such that $f^m(x)\in
S_{\mu_2}(0)$ for any $m\geq 2$. So in this case $k=2$.

If $|f^2(x)|_p={{\alpha^2}\over{|a|_p}}$ then we can continue
the argument and get the following inequality
$$|f^k(x)|_p\leq{{\alpha^2}\over{|a|_p}}.$$
Hence in each step we may have two possibilities:
$|f^k(x)|_p={{\alpha^2}\over{|a|_p}}$ or
$|f^k(x)|_p<{{\alpha^2}\over{|a|_p}}$. In case
$|f^k(x)|_p<{{\alpha^2}\over{|a|_p}}$ there exists $\mu_k$ such
that $f^m(x)\in S_{\mu_k}(0)$  for any $m\geq k$. If
$|f^k(x)|_p={{\alpha^2}\over{|a|_p}}$ for any $k\in \N$ then
$\{f^k(x), k\geq 1\}\subset S_{{\alpha^2}\over{|a|_p}}(0)$.
\end{proof}

We note that $\mathcal P$ defined in (\ref{P}) has the following form $$\mathcal
P=\bigcup_{k=0}^{\infty}{\mathcal P_k},  \ \ \mathcal P_k=\{x\in
\C_p: f^k(x)\in\{\hat x_1, \hat x_2\}\}.$$

\begin{thm}\label{tpk} If $|a|_p>\alpha=\beta$, then
\begin{itemize}
\item[1.] $\mathcal P_k\neq\emptyset, \ \ for \ \ any \ \ k=0,1,2,...
\,.$
\item[2.] $\mathcal P_k\subset S_{r_k}(0)$,   where
$r_k=\alpha\cdot\left(\alpha\over{|a|_p}\right)^{{2^k-1}\over{2^k}}$,
$k=0,1,2,... \,.$
\end{itemize}
\end{thm}
\begin{proof} 1. In case $k=0$ we have $\mathcal P_0=\{\hat x_1, \hat x_2\}\neq\emptyset$.

Assume for $k=n$ that $\mathcal P_n=\{x\in \C_p: f^n(x)\in\{\hat x_1,
\hat x_2\}\}\neq\emptyset$.

Now for $k=n+1$ to prove $\mathcal P_{n+1}=\{x\in \C_p: f^{n+1}(x)\in\{\hat x_1,
\hat x_2\}\}\neq\emptyset$ we have to show that the following equation
has at least one solution:
$$f^{n+1}(x)=\hat x_i, \ \ \mbox{for some} \ \ i=1,2.$$

By our assumption  $\mathcal
P_n\neq\emptyset$ there exists $y\in\mathcal P_n$ such
that $f^n(y)\in \{\hat x_1, \hat x_2\}$. Now we show that there exists $x$ such that $f(x)=y$.
We note that the equation $f(x)=y$ can be written as
\begin{equation}\label{qe}
(a-y)x^2+(b-dy)x-by=0.
\end{equation}
Since $\hat x_1, \hat x_2\in S_\alpha(0)$, by the Lemma \ref{lf2} and the part 1
of Lemma \ref{l3} we know that $S_{{|a|_p}}(0)$ is an invariant, consequently, $\mathcal P\cap S_{|a|_p}(0)=\emptyset$, for ${|a|_p}>\alpha$.
Thus $a\not\in\mathcal P$, hence, $a-y\ne 0$.
Since $\C_p$ is algebraic closed the equation (\ref{qe})
has two solutions, say $x=t_1,t_2$. For $x\in \{t_1,t_2\}$ we get
$$f^{n+1}(x)=f^n(f(x))=f^n(y)\in \{\hat x_1, \hat x_2\}.$$ Hence
 $\mathcal P_{n+1}\ne\emptyset$. Therefore, by induction we get
$$\mathcal P_k\neq\emptyset, \ \ \mbox{for any} \ \ k=0,1,2,... \,.$$

2. We know that $|\hat x_1|_p=|\hat x_2|_p=\alpha$. By condition ${|a|_p}>\alpha$, we get $\alpha>{{\alpha^2}\over{|a|_p}}$.
By (\ref{f2}) and part 2 of Lemma \ref{l3} for $x\in
S_{\alpha}(0)$, $x\neq \hat x_{1,2}$ we have
$$\lim_{n\to\infty}f^n(x)\in S_{{|a|_p}}(0),$$
i.e., $S_{\alpha}(0)\cap\mathcal P=\{\hat x_1, \hat x_2\}=\mathcal
P_0$.
Denoting $r_0=\alpha$ we write $\mathcal P_0\subset S_{r_0}(0)$.

For each  $k=1,2,3,\dots $ we want to find some $r_k$ such that the solution $x$ of $f^k(x)=\hat x_i$, (for some $i=1,2$) belongs to $S_{r_k}(0)$, i.e.,
$x\in S_{r_k}(0)$. By Lemma \ref{lf2} we should have
$$\psi_{\alpha}^k(r_k)=\alpha.$$
Now if we show that the last equation has unique solution $r_k$ for each $k$, then
we get
$$\mathcal P_k=\{x\in \C_p: f^k(x)\in\{\hat x_1, \hat x_2\}\}\subset
S_{r_k}(0).$$

By parts 1 and 3 of Lemma \ref{l3} we have
${{\alpha^2}\over{|a|_p}}<r_k\leq \alpha$.
Moreover, we have $r_0=\alpha$ and
${{\alpha^2}\over{|a|_p}}<r_k<\alpha$ for each $k=1,2,...$
For such $r_k$, by definition of $\psi_{\alpha}(r)$,  we have
$$\psi_{\alpha}(r_k)={{{|a|_p} r_k^2}\over{\alpha^2}}.$$
Thus $\psi_{\alpha}^k(r_k)=\alpha$ has the form
$$\psi_{\alpha}^k(r_k)={{{|a|_p}^{2^k-1}}\over{\alpha^{2(2^k-1)}}}r^{2^k}_k=\alpha$$
consequently
$$r^{2^k}_k=\alpha^{2^k}\cdot\left[\left({\alpha\over{|a|_p}}\right)^{{2^k-1}\over{2^k}}\right]^{2^k}.$$
Taking  $2^k$-root
we obtain unique positive solution:
$r_k=\alpha\cdot\left({\alpha\over{|a|_p}}\right)^{{2^k-1}\over{2^k}}$.
\end{proof}

\subsection{Case: $|a|_p<\alpha<\beta$.}

\begin{rk} In this case Theorem \ref{t4} gives the following character of the dynamical system:
the open ball with radius $\alpha$ and center $x_1$ is the maximal Siegel disk for fixed point $x_1$. The fixed point $x_2$ is repeller and the inequality $|f(x)-x_2|_p>|x-x_2|_p$ holds for all $x\in U_{\beta}(x_2)\subset S_{\beta}(0), \, x\neq x_2$. If $|x|_p>\beta$, then trajectory of the point $x$ is subset of the maximal Siegel disk of the fixed point $x_1$.
\end{rk}

\begin{lemma}\label{l4} If $|a|_p<\alpha<\beta$, then the dynamical system generated by $\varphi_{\alpha,\beta}(r)$ has the following properties:
\begin{itemize}
\item[1.] ${\rm Fix}(\varphi_{\alpha,\beta})=\{r: 0\leq r<\alpha\}\cup\{\alpha:\, if \,
\check\alpha=\alpha\}\cup\{\beta:\, if \, \check\beta=\beta\}$.
\item[2.] If $\alpha<r<\beta$, then $\varphi_{\alpha,\beta}(r)=\alpha.$
\item[3.] If $r>\beta$, then
$$\varphi_{\alpha,\beta}^n(r)=\left\{\begin{array}{lll}
{{\alpha\beta}\over r}, \ \ \mbox{for all} \ \ \beta<r<{{\alpha\beta}\over {|a|_p}}\\[2mm]
\check{a}, \ \ \mbox{for} \ \ r={{\alpha\beta}\over {|a|_p}}\\[2mm]
{|a|_p}, \ \ \ \ \mbox{for all} \ \ r>{{\alpha\beta}\over {|a|_p}}
\end{array}
\right.$$ $\mbox{for any} \ \ n\geq 1$.
\item[4.] Let $r=\alpha$.
\begin{itemize}
\item[4.1)] If $\alpha<\check\alpha<\beta$, then
$\varphi_{\alpha,\beta}^2(\alpha)=\alpha$.
\item[4.2)] If $\check\alpha=\beta$, then $\varphi_{\alpha,\beta}(\alpha)=\beta$.
\item[4.3)] If $\check\alpha>\beta$, then
$$\varphi_{\alpha,\beta}^n(\alpha)=\left\{\begin{array}{lll}
{{\alpha\beta}\over \check\alpha}, \ \ \mbox{if} \ \ \beta<\check\alpha<{{\alpha\beta}\over {|a|_p}}\\[2mm]
\check{a}, \ \ \mbox{if} \ \ \check\alpha={{\alpha\beta}\over {|a|_p}}\\[2mm]
{|a|_p}, \ \ \ \ \mbox{if} \ \ \check\alpha>{{\alpha\beta}\over {|a|_p}}
\end{array}
\right.$$ $\mbox{for any} \ \ n\geq 2$.
\end{itemize}
\item[5.] Let $r=\beta$.
\begin{itemize}
\item[5.1)] If $\alpha<\check\beta<\beta$, then
$\varphi_{\alpha,\beta}^2(\beta)=\alpha$.
\item[5.2)] If $\check\beta>\beta$, then
$$\varphi_{\alpha,\beta}^n(\beta)=\left\{\begin{array}{lll}
{{\alpha\beta}\over \check\beta}, \ \ \mbox{if} \ \ \beta<\check\beta<{{\alpha\beta}\over {|a|_p}}\\[2mm]
\check{a}, \ \ \mbox{if} \ \ \check\beta={{\alpha\beta}\over {|a|_p}}\\[2mm]
{|a|_p}, \ \ \ \ \mbox{if} \ \ \check\beta>{{\alpha\beta}\over {|a|_p}}
\end{array}
\right.$$ $\mbox{for any} \ \ n\geq 2$.
\end{itemize}
\end{itemize}
\end{lemma}

\begin{proof} 1. This is the result of a simple analysis
of the equation $\varphi_{\alpha,\beta}(r)=r$.

2. If there is $\alpha<r<\beta$, then function $\varphi_{\alpha,\beta}$ will be $\varphi_{\alpha,\beta}(r)=\alpha$ by definition.

3. If $r>\beta$, then
$$\varphi_{\alpha,\beta}(r)=\left\{\begin{array}{lll}
{\alpha\beta\over r}, \ \ {\rm if} \ \ \beta<r<{{\alpha\beta}\over {|a|_p}}\\[2mm]
\check a, \ \ {\rm if} \ \ r={{\alpha\beta}\over {|a|_p}}\\[2mm]
|a|_p, \ \ {\rm if} \ \ r>{{\alpha\beta}\over {|a|_p}}.
\end{array}
\right.
$$
Consequently,
$$\beta<r<{{\alpha\beta}\over {|a|_p}} \ \ \Rightarrow \ \ {|a|_p}<{{\alpha\beta}\over r}<\alpha \ \ \Rightarrow \ \ \varphi_{\alpha,\beta}(r)<\alpha.$$
If $r\geq{{\alpha\beta}\over {|a|_p}}$, then by
$\check a\leq{|a|_p}<\alpha$ we have
$\varphi_{\alpha,\beta}(r)<\alpha$.
Thus $\varphi_{\alpha,\beta}(\varphi_{\alpha,\beta}(r))=\varphi_{\alpha,\beta}(r)$,
i.e., $\varphi_{\alpha,\beta}(r)$ is a fixed point of
$\varphi_{\alpha,\beta}$ for any $r>\beta$. Consequently, for each $n\geq 1$ we have
$$\varphi_{\alpha,\beta}^n(r)=\left\{\begin{array}{lll}
{\alpha\beta\over r}, \ \ {\rm if} \ \ \beta<r<{{\alpha\beta}\over{|a|_p}}\\[2mm]
\check a, \ \ {\rm if} \ \ r={{\alpha\beta}\over{|a|_p}}\\[2mm]
{|a|_p}, \ \ {\rm if} \ \ r>{{\alpha\beta}\over{|a|_p}}.
\end{array}
\right.$$

The parts 4 and 5 easily follows from the parts 1, 2 and 3.
\end{proof}

If $\alpha<\beta$, then we note that $\mathcal P$ has the
following form $\mathcal P=\mathcal P_{\alpha}\cup\mathcal
P_{\beta}$, where $$\mathcal P_{\alpha}=\{x\in \C_p: \exists n\in
\N\cup\{0\}, f^n(x)=\hat x_1\} \ \ {\rm and} \ \ \mathcal
P_{\beta}=\{x\in \C_p: \exists n\in \N\cup\{0\}, f^n(x)=\hat x_2\}.$$

For $|a|_p<\alpha<\beta$ denote the following
$$\check\alpha(x)=|f(x)|_p, \ \ {\rm if} \ \ x\in
S_{\alpha}(0)\setminus\{\hat x_1\}; \ \ \ \check\beta(x)=|f(x)|_p, \ \
{\rm if} \ \ x\in S_{\beta}(0)\setminus\{\hat x_2\};$$
$$\check a(x)=|f(x)|_p, \ \ {\rm if} \ \ x\in
S_{{\alpha\beta}\over {|a|_p}}(0).$$

Then using Lemma \ref{lf2} and Lemma \ref{l4} we obtain the
following

\begin{thm}\label{t4} If $|a|_p<\alpha<\beta$ and $x\in S_r(0)\setminus\mathcal P$, then the $p$-adic dynamical
system generated by function (\ref{fd}) has the following properties:
\begin{itemize}
\item[1.] $SI(x_1)=U_{\alpha}(0)$.
\item[2.] If $\alpha<r<\beta$, then $f(x)\in S_{\alpha}(0)$.
\item[3.] Let $r>\beta$, then $$f^n(x)\in\left\{\begin{array}{lll}
S_{\alpha\beta\over r}(0), \ \ \mbox{for all} \ \ \alpha<r<{{\alpha\beta}\over {|a|_p}}\\[2mm]
S_{\check a(x)}(0), \ \ \mbox{for} \ \ r={{\alpha\beta}\over {|a|_p}}\\[2mm]
S_{|a|_p}(0), \ \ \mbox{for all} \ \ r>{{\alpha\beta}\over{|a|_p}},
\end{array}
\right.$$ for any $n\geq 1$.
\item[4.] Let $x\in S_{\alpha}(0)\setminus\mathcal P$.
\begin{itemize}
\item[4.1)] If $\check\alpha(x)=\alpha$, then $f(x)\in
S_{\alpha}(0)$.
\item[4.2)] If $\alpha<\check\alpha(x)<\beta$, then $f^2(x)\in
S_{\alpha}(0)$.
\item[4.3)] If $\check\alpha(x)=\beta$, then $f(x)\in
S_{\beta}(0)$.
\item[4.4)] If $\check\alpha(x)>\beta$, then $$f^n(x)\in\left\{\begin{array}{lll}
S_{\alpha\beta\over{\check\alpha(x)}}(0), \ \ \mbox{if} \ \ \alpha<{\check\alpha(x)}<{{\alpha\beta}\over {|a|_p}}\\[2mm]
S_{\check a(f(x))}(0), \ \ \mbox{if} \ \ {\check\alpha(x)}={{\alpha\beta}\over {|a|_p}}\\[2mm]
S_{|a|_p}(0), \ \ \mbox{if} \ \
{\check\alpha(x)}>{{\alpha\beta}\over {|a|_p}}
\end{array}
\right.$$ for any $n\geq 2$.
\end{itemize}
\item[5.] Let $x\in S_{\beta}(0)\setminus\mathcal P$.
\begin{itemize}
\item[5.1)] If $\check\beta(x)=\alpha$, then $f(x)\in
S_{\alpha}(0)$.
\item[5.2)] If $\alpha<\check\beta(x)<\beta$, then $f^2(x)\in
S_{\alpha}(0)$.
\item[5.3)] If $\check\beta(x)=\beta$, then $f(x)\in
S_{\beta}(0)$.
\item[5.4)] If $\check\beta(x)>\beta$, then $$f^n(x)\in\left\{\begin{array}{lll}
S_{\alpha\beta\over{\check\beta(x)}}(0), \ \ \mbox{if} \ \ \alpha<{\check\beta(x)}<{{\alpha\beta}\over {|a|_p}}\\[2mm]
S_{\check a(f(x))}(0), \ \ \mbox{if} \ \ {\check\beta(x)}={{\alpha\beta}\over {|a|_p}}\\[2mm]
S_{|a|_p}(0), \ \ \mbox{if} \ \
{\check\beta(x)}>{{\alpha\beta}\over {|a|_p}}
\end{array}
\right.$$ for any $n\geq 2$.
\end{itemize}
\item[6.]
\begin{itemize}
\item[6.1)] $x_2\in S_{\beta}(0)$.
\item[6.2)] The fixed point $x_2$ is repeller and the inequality $|f(x)-x_2|_p>|x-x_2|_p$ holds
for $x\in U_{\beta}(x_2)$, $x\neq x_2$.
\end{itemize}
\end{itemize}
\end{thm}
\begin{proof}
The parts 1-5 of this theorem can be shown by Lemma \ref{lf2} and Lemma \ref{l4}.

6. In this case we note that $|a|_p<\alpha<\beta=|d|_p$. Then $|x_2|_p=|a-d|_p=\beta$, i.e., $x_2\in S_{\beta}(0)$, and we have
$$|f'(x_2)|_p=\left|{{b+(a-d)d}\over{b+(a-d)a}}\right|_p={{\max\{\alpha\beta,\beta^2\}}\over{\max\{\alpha\beta,|a|_p\beta\}}}={\beta\over\alpha}>1.$$
So, $x_2$ is repeller fixed point for $f$.

In this case, in (\ref{fx2}) we have $|x_2-\hat x_1|_p=\beta$ and $|x_2-\hat x_2|_p=\alpha$. Indeed, $|x_2-\hat x_1|_p=\max\{\alpha,\beta\}=\beta$ and $|x_2-\hat x_2|_p=|a+\hat x_1|_p=\max\{|a|_p,\alpha\}=\alpha$. Moreover, by formula (\ref{fx2}) we have
$$|f(x)-x_2|_p=\left\{\begin{array}{lllll}
{\beta\over\alpha}|x-x_2|_p, \ \ {\rm if} \ \ |x-x_2|_p<\alpha \\[2mm]
\geq\beta, \ \ {\rm if} \ \ |x-x_2|_p=\alpha \\[2mm]
\beta, \ \ \ \ {\rm if} \ \ \alpha<|x-x_2|_p<\beta \\[2mm]
\beta^0, \ \ \ \ {\rm if} \ \ |x-x_2|_p=\beta \\[2mm]
\beta, \ \ \ \ {\rm if} \ \ |x-x_2|_p>\beta,
\end{array}
\right.$$
where $\beta^0>0$.

That is the inequality $|f(x)-x_2|_p>|x-x_2|_p$ satisfied
for $|x-x_2|_p<\beta$, i.e., for $x\in U_{\beta}(x_2)$, $x\neq x_2$.
\end{proof}

\subsection{Case: $|a|_p=\alpha<\beta$.}

\begin{rk} In this case Theorem \ref{t5} gives the following character of the dynamical system:
 the open ball with radius $\alpha$ and center $x_1$ is the maximal Siegel disk for fixed point $x_1$.
 The fixed point $x_2$ is repeller and the inequality $|f(x)-x_2|_p>|x-x_2|_p$ holds for all $x\in U_{\beta}(x_2)\subset S_{\beta}(0), \, x\neq x_2$.
\end{rk}

\begin{lemma}\label{l5} If $|a|_p=\alpha<\beta$, then the
dynamical system generated by $\phi_{\alpha,\beta}(r)$ has the following properties:
\begin{itemize}
\item[1.] ${\rm Fix}(\phi_{\alpha,\beta})=\{r: 0\leq r<\alpha\}\cup\{\alpha:\, if \,
\tilde\alpha=\alpha\}\cup\{\beta: \, if \, \tilde\beta=\beta\}$.
\item[2.] If $r>\alpha$ and $r\neq\beta$, then
$\phi_{\alpha,\beta}(r)=\alpha.$
\item[3.] Let $r=\alpha$ and $\tilde\alpha>\alpha$.
\begin{itemize}
\item[3.1)] If $\tilde\alpha\neq\beta$, then $\phi_{\alpha,\beta}^2(\alpha)=\alpha.$
\item[3.2)] If $\tilde\alpha=\beta$, then $\phi_{\alpha,\beta}(\alpha)=\beta.$
\end{itemize}
\item[4.] If $r=\beta$
\begin{itemize}
\item[4.1)] If $\tilde\beta<\alpha$, then
$\phi_{\alpha,\beta}^n(\beta)=\tilde\beta$ for all $n\geq 1$.
\item[4.2)] If $\tilde\beta=\alpha$, then $\phi_{\alpha,\beta}(\beta)=\alpha$.
\item[4.3)] If $\tilde\beta>\alpha$ and $\tilde\beta\neq\beta$, then
$\phi_{\alpha,\beta}^2(\beta)=\alpha.$
\end{itemize}
\end{itemize}
\end{lemma}
\begin{proof} Is similar to the proof of previous lemmas.
\end{proof}

For $|a|_p=\alpha<\beta$ denote the following
$$\tilde\alpha(x)=|f(x)|_p, \ \ {\rm if} \ \ x\in
S_{\alpha}(0)\setminus\{\hat x_1\}; \ \ \ \tilde\beta(x)=|f(x)|_p, \ \
{\rm if} \ \ x\in S_{\beta}(0)\setminus\{\hat x_2\}.$$

By Lemma \ref{lf2} and Lemma \ref{l5} we get

\begin{thm}\label{t5} If $|a|_p=\alpha<\beta$ and $x\in S_r(0)\setminus\mathcal P$, then
 the $p$-adic dynamical system generated by function (\ref{fd}) has the following properties:
\begin{itemize}
\item[1.] $SI(x_1)=U_{\alpha}(0)$.
\item[2.] If $r>\alpha$ and $r\ne\beta$, then
$f(x)\in S_{\alpha}(0)$.
\item[3.] Let $x\in S_{\alpha}(0)\setminus\mathcal P$.
\begin{itemize}
\item[3.1)] If $\tilde\alpha(x)=\alpha$, then $f(x)\in S_{\alpha}(0)$.
\item[3.2)] If $\tilde\alpha(x)>\alpha$ and $\tilde\alpha(x)\ne\beta$, then $f^2(x)\in S_{\alpha}(0)$.
\item[3.3)] If $\tilde\alpha(x)=\beta$, then $f(x)\in S_{\beta}(0)$.
\end{itemize}
\item[4.] Let $x\in S_{\beta}(0)\setminus\mathcal P$.
\begin{itemize}
\item[4.1)] If $\tilde\beta(x)<\alpha$, then $f^n(x)\in S_{\tilde\beta(x)}(0)$ for any $n\geq 1$.
\item[4.2)] If $\tilde\beta(x)=\alpha$, then $f(x)\in S_{\alpha}(0)$.
\item[4.3)] If $\tilde\beta(x)>\alpha$ and $\tilde\beta(x)\ne\beta$, then $f^2(x)\in S_{\alpha}(0)$.
\item[4.4)] If $\tilde\beta(x)=\beta$, then $f(x)\in S_{\beta}(0)$.
\end{itemize}
\item[5.]
\begin{itemize}
\item[5.1)] $x_2\in S_{\beta}(0)$.
\item[5.2)] The fixed point $x_2$ is repeller and the inequality $|f(x)-x_2|_p>|x-x_2|_p$ satisfied
for $x\in U_{\beta}(x_2)$, $x\neq x_2$.
\end{itemize}
\end{itemize}
\end{thm}
\begin{proof}
The parts 1-4 of this theorem are easily obtained from Lemma \ref{lf2} and Lemma \ref{l5}.

5. In this case we note that $|a|_p=\alpha<\beta=|d|_p$. Then $|x_2|_p=|a-d|_p=\beta$, i.e. $x_2\in S_{\beta}(0)$ and we have
$$|f'(x_2)|_p=\left|{{b+(a-d)d}\over{b+(a-d)a}}\right|_p\geq{\beta\over\alpha}>1.$$
So, $x_2$ is repeller fixed point for $f$.

In this case, in (\ref{fx2}) we have $|x_2-\hat x_1|_p=\beta$ and $|x_2-\hat x_2|_p\leq\alpha$. Indeed, $|x_2-\hat x_2|_p=|a+\hat x_1|_p\leq\alpha$.
Moreover, by formula (\ref{fx2}) we have
$$|f(x)-x_2|_p=\left\{\begin{array}{lllll}
{\beta\over{|x_2-\hat x_2|_p}}|x-x_2|_p, \ \ {\rm if} \ \ |x-x_2|_p<|x_2-\hat x_2|_p \\[2mm]
\geq\beta, \ \ {\rm if} \ \ |x-x_2|_p=|x_2-\hat x_2|_p \\[2mm]
\beta, \ \ \ \ {\rm if} \ \ |x_2-\hat x_2|_p<|x-x_2|_p<\beta \\[2mm]
\beta_1, \ \ \ \ {\rm if} \ \ |x-x_2|_p=\beta \\[2mm]
\beta, \ \ \ \ {\rm if} \ \ |x-x_2|_p>\beta,
\end{array}
\right.$$
where $\beta_1>0$.

That is the inequality $|f(x)-x_2|_p>|x-x_2|_p$ satisfied
for $0<|x-x_2|_p<\beta$, i.e., for $x\in U_{\beta}(x_2)$, $x\neq x_2$.
\end{proof}

\subsection{Case: $\alpha<|a|_p<\beta$.}

\begin{rk} In this case the open ball with radius $\alpha$ and center $x_1$ is the maximal Siegel disk for fixed point $x_1$. The fixed point $x_2$ is repeller and the inequality $|f(x)-x_2|_p>|x-x_2|_p$ holds for all $x\in U_{\beta}(x_2)\subset S_{\beta}(0), \, x\neq x_2$.
\end{rk}

\begin{lemma}\label{l6} If $\alpha<|a|_p<\beta$, then the dynamical system generated by $\psi_{\alpha,\beta}(r)$ has the following properties:
\begin{itemize}
\item[1.] ${\rm Fix}(\psi_{\alpha,\beta})=\{r: 0\leq r<\alpha\}\cup\{\alpha:\, if \,
\breve\alpha=\alpha\}\cup\{\beta: \, if \, \breve\beta=\beta\}$.
\item[2.] If $r>\beta$, then $\psi_{\alpha,\beta}(r)=|a|_p\in(\alpha,\beta)$.
\item[3.] Let $r\in(\alpha,\beta)$.
\begin{itemize}
\item[3.1)] If $r\in B$ and $\breve a<\alpha$, then $$\lim_{n\to\infty}\psi_{\alpha,\beta}^n(r)=\breve a.$$
\item[3.2)] If $r\in B$ and $\breve a=\alpha$, then there exists $n\in N$ such that $\psi_{\alpha,\beta}^n(r)=\alpha$.
\item[3.3)] If $r\not\in B$, then there exists $n\in N$ such that $\psi_{\alpha,\beta}^n(r)=\alpha$.
\end{itemize}
\item[4.] If $r=\beta$ and $\breve\beta\neq\beta$, then $\psi_{\alpha,\beta}(\beta)>\beta$.
\item[5.] If $r=\alpha$ and $\breve\alpha\neq\alpha$, then $\psi_{\alpha,\beta}(\alpha)>\alpha$.
\end{itemize}
\end{lemma}
where $B=\left\{r_n| \ \ r_n=\alpha{{\beta^{n+1}}\over{|a|_p^{n+1}}}, \, n\in\{0\}\cup\left(N\cap\left(0, \log_{{\beta}\over{|a|_p}}{{|a|_p}\over{\alpha}}\right)\right)\right\}$.
\begin{proof}
Parts 1-2 are straightforward.

3. By definition of $\psi_{\alpha,\beta}(r)$, for
$r\in(\alpha,\beta)$  we have
$$\psi_{\alpha,\beta}(r)=\left\{\begin{array}{lll}
\alpha, \ \ {\rm if} \ \ \alpha<r<{{\alpha\beta}\over {|a|_p}}\\[2mm]
\breve a, \ \ {\rm if} \ \ r={{\alpha\beta}\over {|a|_p}}\\[2mm]
{{|a|_pr}\over\beta}, \ \ {\rm if} \ \ {{\alpha\beta}\over {|a|_p}}<r<\beta.
\end{array}
\right.
$$
where $\breve a\leq\alpha$.

3.1) If $\breve a<\alpha$, then from part 1 of this lemma $\breve a$ is fixed point for $\psi_{\alpha,\beta}$.

Now we consider the set
$$B=\left\{r| \ \ \exists n\in N\cup\{0\}, \, \psi_{\alpha,\beta}^n(r)={{\alpha\beta}\over{|a|_p}}\right\}$$ $$=\left\{r_n| \ \ r_n=\alpha{{\beta^{n+1}}\over{|a|_p^{n+1}}}, \, n\in\{0\}\cup\left(N\cap\left(0, \log_{{\beta}\over{|a|_p}}{{|a|_p}\over{\alpha}}\right)\right)\right\}\subset\left[{{\alpha\beta}\over{|a|_p}}, \beta\right).$$

It is easy to see that if $\breve a<\alpha$, then $$\lim_{n\to\infty}\psi_{\alpha,\beta}^n(r)=\breve a$$ for $r\in B$.

3.2) If $\breve a=\alpha$, then $\psi_{\alpha,\beta}^n(r)={{\alpha\beta}\over{|a|_p}}$ and $\psi_{\alpha,\beta}^{n+1}(r)=\alpha$ for $r=r_n\in B$.

3.3) Let $r\in(\alpha,\beta)\setminus B$. Then $r\in\left(\alpha, {{\alpha\beta}\over{|a|_p}}\right)$ or $r\in\left[{{\alpha\beta}\over{|a|_p}}, \beta\right)\setminus B$.

If $\alpha<r<{{\alpha\beta}\over{|a|_p}}$, then by definition of $\psi_{\alpha,\beta}$ we have $\psi_{\alpha,\beta}(r)=\alpha$, i.e. $n=1$.

If $r\in\left[{{\alpha\beta}\over{|a|_p}}, \beta\right)\setminus B$, then there exists $r_k\in B$ such that $r_k<r<r_{k+1}$ for $k<\lfloor \log_{{\beta}\over{|a|_p}}{{|a|_p}\over{\alpha}}\rfloor$ or $r_k<r<\beta$ at $k=\lfloor\log_{{\beta}\over{|a|_p}}{{|a|_p}\over{\alpha}}\rfloor$. Moreover, $r_{k-1}<\psi_{\alpha,\beta}(r)<r_k$. Indeed, for $k<\lfloor\log_{{\beta}\over{|a|_p}}{{|a|_p}\over{\alpha}}\rfloor$ we have $$\alpha{{\beta^{k+1}}\over{|a|_p^{k+1}}}<r<\alpha{{\beta^{k+2}}\over{|a|_p^{k+2}}} \ \ \mbox{and} \ \
\alpha{{\beta^{k}}\over{|a|_p^{k}}}<\psi_{\alpha,\beta}(r)<\alpha{{\beta^{k+1}}\over{|a|_p^{k+1}}};$$
for $k=\lfloor\log_{{\beta}\over{|a|_p}}{{|a|_p}\over{\alpha}}\rfloor$ we have
$$\alpha{{\beta^{k+1}}\over{|a|_p^{k+1}}}<r<\beta \ \ \mbox{and} \ \ \alpha{{\beta^{k}}\over{|a|_p^{k}}}<\psi_{\alpha,\beta}(r)<|a|_p<\alpha{{\beta^{k+1}}\over{|a|_p^{k+1}}}$$
because, $\psi_{\alpha,\beta}(r)={{|a|_pr}\over\beta}$ and $\beta<\alpha{{\beta^{k+2}}\over{|a|_p^{k+2}}}$.

Consequently, we get ${{\alpha\beta}\over{|a|_p}}=r_{0}<\psi_{\alpha,\beta}^{k}(r)<r_1$ and $\alpha<\psi_{\alpha,\beta}^{k+1}(r)<{{\alpha\beta}\over{|a|_p}}$. By definition of $\psi_{\alpha,\beta}$ we have $\psi_{\alpha,\beta}^{k+2}(r)=\alpha$, i.e. $n=k+2$.

4.-5. These results are directly presented in the definition of function $\psi_{\alpha,\beta}$.
\end{proof}

For $\alpha<|a|_p<\beta$ denote the following
$$\breve\alpha(x)=|f(x)|_p, \ \ {\rm if} \ \ x\in
S_{\alpha}(0)\setminus\{\hat x_1\};$$ $$\breve\beta(x)=|f(x)|_p, \ \
{\rm if} \ \ x\in S_{\beta}(0)\setminus\{\hat x_2\};$$ $$\breve a(x)=|f(x)|_p, \ \
{\rm if} \ \ x\in S_{{\alpha\beta}\over{|a|_p}}(0).$$

Using Lemma \ref{lf2} and Lemma \ref{l6} we get

\begin{thm}\label{t6} If $\alpha<|a|_p<\beta$ and $x\in S_r(0)\setminus\mathcal P$, then
 the $p$-adic dynamical system generated by function (\ref{fd}) has the following properties:
\begin{itemize}
\item[1.] $SI(x_1)=U_{\alpha}(0)$.
\item[2.] If $r>\beta$, then $f(x)\in S_{|a|_p}(0)$.
\item[3.] Let $r\in(\alpha,\beta)$.
\begin{itemize}
\item[3.1)] If $r\in B$ and $\breve a(x)<\alpha$, then there exists $n_0\in N$ such that $f^n(x)\in S_{\breve a(x)}(0)$ for all $n>n_0$.
\item[3.2)] If $r\in B$ and $\breve a(x)=\alpha$, then there exists $n\in N$ such that $f^n(x)\in S_{\alpha}(0)$.
\item[3.3)] If $r\not\in B$, then there exists $n\in N$ such that $f^n(x)\in S_{\alpha}(0)$.
\end{itemize}
\item[4.] If $r=\beta$ and $\breve\beta(x)\neq\beta$, then $f(x)\not\in V_{\beta}(0)$.
\item[5.] If $r=\alpha$ and $\breve\alpha(x)\neq\alpha$, then $f(x)\not\in V_{\alpha}(0)$.
\item[6.]
\begin{itemize}
\item[6.1)] $x_2\in S_{\beta}(0)$.
\item[6.2)] The fixed point $x_2$ is repeller and the inequality $|f(x)-x_2|_p>|x-x_2|_p$ satisfied
for $x\in U_{\beta}(x_2)$, $x\neq x_2$.
\end{itemize}
\end{itemize}
\end{thm}
\begin{proof}
The parts 1-5 of this theorem are easily obtained by Lemma \ref{lf2} and Lemma \ref{l6}.

6. In this case we note that $|a|_p<\beta=|d|_p$. Then $|x_2|_p=|a-d|_p=\beta$, i.e. $x_2\in S_{\beta}(0)$ and we have
$$|f'(x_2)|_p=\left|{{b+(a-d)d}\over{b+(a-d)a}}\right|_p={\beta\over{|a|_p}}>1.$$
So, $x_2$ is repeller fixed point for $f$.

In this case, in (\ref{fx2}) we have $|x_2-\hat x_1|_p=\beta$ and $|x_2-\hat x_2|_p=|a|_p$. Indeed, $|x_2-\hat x_2|_p=|a+\hat x_1|_p=|a|_p$.
Moreover, by formula (\ref{fx2}) we have
$$|f(x)-x_2|_p=\left\{\begin{array}{lllll}
{\beta\over{|a|_p}}|x-x_2|_p, \ \ {\rm if} \ \ |x-x_2|_p<|a|_p \\[2mm]
\geq\beta, \ \ {\rm if} \ \ |x-x_2|_p=|a|_p \\[2mm]
\beta, \ \ \ \ {\rm if} \ \ |a|_p<|x-x_2|_p<\beta \\[2mm]
\beta_2, \ \ \ \ {\rm if} \ \ |x-x_2|_p=\beta \\[2mm]
\beta, \ \ \ \ {\rm if} \ \ |x-x_2|_p>\beta,
\end{array}
\right.$$
where $\beta_2>0$.

That is the inequality $|f(x)-x_2|_p>|x-x_2|_p$ satisfied
for $0<|x-x_2|_p<\beta$, i.e., for $x\in U_{\beta}(x_2)$, $x\neq x_2$.
\end{proof}

\subsection{Case: $\alpha<|a|_p=\beta$.}

\begin{rk} In this case Theorem \ref{t7} gives the following character of the dynamical system:
the open ball with radius $\alpha$ and center $x_1$ is the maximal Siegel disk for fixed point $x_1$,
moreover if $\alpha<r<\beta$, then the sphere $S_r(0)$ is invariant for $f$.
If fixed point $x_2$ is indifferent then it have the Siegel disk and

$1. \, SI(x_2)=SI(x_1), \, if \, |a-d|_p<\alpha;$

$2. \, SI(x_2)\cap SI(x_1)=\emptyset, \, if \, |a-d|_p>\alpha.$
\end{rk}

\begin{lemma}\label{l7} If $\alpha<|a|_p=\beta$, then the dynamical system generated by $\eta_{\alpha,\beta}(r)$ has the following properties:
\begin{itemize}
\item[1.] ${\rm Fix}(\eta_{\alpha,\beta})=\{r: 0\leq r<\alpha\ \, \mbox{or} \ \ \alpha<r<\beta\}\cup\{\alpha:\, if \,
\acute\alpha=\alpha\}\cup\{\beta:\, if \, \acute\beta=\beta\}$.
\item[2.] If $r>\beta$, then $\eta_{\alpha,\beta}(r)=\beta.$
\item[3.] Let $r=\alpha$.
\begin{itemize}
\item[3.1)] If $\acute\alpha<\beta$, then
$\eta^n_{\alpha,\beta}(\alpha)=\acute\alpha$, for all $n\geq 1$.
\item[3.2)] If $\acute\alpha=\beta$, then $\eta_{\alpha,\beta}(\alpha)=\beta.$
\item[3.3)] If $\acute\alpha>\beta$, then $\eta^2_{\alpha,\beta}(\alpha)=\beta.$
\end{itemize}
\item[4.] If $r=\beta$ and $\acute\beta>\beta$, then $\eta^2_{\alpha,\beta}(\beta)=\beta.$
\end{itemize}
\end{lemma}
\begin{proof} Similar to proofs of the previous lemmas.
\end{proof}

For $\alpha<|a|_p=\beta$ denote the following
$$\acute\alpha(x)=|f(x)|_p, \ \ {\rm if} \ \ x\in
S_{\alpha}(0)\setminus\{\hat x_1\};$$ $$\acute\beta(x)=|f(x)|_p, \ \
{\rm if} \ \ x\in S_{\beta}(0)\setminus\{\hat x_2\}.$$

Using Lemma \ref{lf2} and Lemma \ref{l7} we get

 \begin{thm}\label{t7} If $\alpha<|a|_p=\beta$, then
 the $p$-adic dynamical system generated by function (\ref{fd}) has the following properties:
\begin{itemize}
\item[1.]
\begin{itemize}
\item[1.1)] $SI(x_1)=U_{\alpha}(0)$.
\item[1.2)] If $\alpha<r<\beta$, then $S_r(0)$ is an invariant for $f$.
\end{itemize}
\item[2.] If $r>\beta$, then $f(S_r(0)\setminus\mathcal P)\subset S_{\beta}(0)$.
\item[3.] Let $x\in S_{\alpha}(0)\setminus\mathcal P$.
\begin{itemize}
\item[3.1)] If $\acute\alpha(x)=\alpha$, then $f(x)\in S_{\alpha}(0)$.
\item[3.2)] If $\acute\alpha(x)\neq\alpha$ and $\acute\alpha(x)<\beta$, then $f^n(x)\in S_{\acute\alpha(x)}(0)$ for any $n\geq 1$.
\item[3.3)] If $\acute\alpha(x)=\beta$, then $f(x)\in S_{\beta}(0)$.
\item[3.4)] If $\acute\alpha(x)>\beta$, then $f^2(x)\in S_{\beta}(0)$.
\end{itemize}
\item[4.] Let $x\in S_{\beta}(0)\setminus\mathcal P$.
\begin{itemize}
\item[4.1)] If $\acute\beta(x)=\beta$, then $f(x)\in S_{\beta}(0)$.
\item[4.2)] If $\acute\beta(x)>\beta$, then $f^2(x)\in S_{\beta}(0)$.
\end{itemize}
\item[5.] $x_2\in V_{\beta}(0)$ and if $|a-d|_p\neq\alpha$, then $x_2$ is indifferent fixed point for $f$.
\begin{itemize}
\item[5.1)] If $|a-d|_p<\alpha$, then $SI(x_2)=SI(x_1)$.
\item[5.2)] If $|a-d|_p>\alpha$, then $SI(x_2)=U_{|x_2|_p}(x_2)$.
\end{itemize}
\end{itemize}
\end{thm}
\begin{proof}
The parts 1-4 of this theorem are consequences of Lemma \ref{lf2} and Lemma \ref{l7}.

5. We note that $|b|_p=\alpha\beta$ and $|a|_p=|d|_p=\beta$. Then $|x_2|_p=|a-d|_p\leq\beta$, i.e. $x_2\in V_{\beta}(0)$. If $|a-d|_p\neq\alpha$, then we have
$$|f'(x_2)|_p=\left|{{b+(a-d)d}\over{b+(a-d)a}}\right|_p={{\max\{\alpha\beta, \, |a-d|_p\beta\}\over{\max\{\alpha\beta, \, |a-d|_p\beta\}}}}=1.$$
So, $x_2$ is indifferent fixed point for $f$.

Let $|a-d|_p<\alpha$ and $x\in S_{\rho}(x_2)$, $\rho<\alpha$. In this case, in (\ref{fx2}) we have $|x_2-\hat x_1|_p=\alpha$ and $|x_2-\hat x_2|_p=\beta$.
Moreover, by formula (\ref{fx2}) we have
$$|f(x)-x_2|_p=|x-x_2|_p\cdot{{|d(x-x_2)+dx_2+b|_p}\over{|(x-x_2)+(x_2-\hat x_1)|_p|(x-x_2)+(x_2-\hat x_2)|_p}}=|x-x_2|_p.$$
Consequently, $f(S_{\rho}(x_2))\subset S_{\rho}(x_2)$ for all $\rho<\alpha$, i.e., $U_{\alpha}(x_2)\subset SI(x_2)$.
We have $|x_2|_p<\alpha$, then $U_{\alpha}(x_2)=U_{\alpha}(0)=SI(x_1)$ and $S_{\alpha}(x_2)=S_{\alpha}(0)$. From part 3 of this theorem $f(S_{\alpha}(0))\nsubseteq S_{\alpha}(0)$. Consequently $SI(x_2)=U_{\alpha}(x_2)=SI(x_1)$.

 Let $\alpha<|a-d|_p\leq\beta$. Then we have  $|x_2-\hat x_1|_p=|x_2|_p$ and $|x_2-\hat x_2|_p=|a+x_1|_p=\beta$. If $|x-x_2|_p=\rho<|x_2|_p$, then by formula (\ref{fx2}) we have
$$|f(x)-x_2|_p=|x-x_2|_p.$$
Consequently, $f(S_{\rho}(x_2))\subset S_{\rho}(x_2)$ for all $\rho<|x_2|_p$, i.e., $U_{|x_2|_p}(x_2)\subset SI(x_2)$.

From equality (\ref{fx2}), it follows that $f(S_{|x_2|_p}(x_2))\not\subset S_{|x_2|_p}(x_2)$. So, we have $SI(x_2)=U_{|x_2|_p}(x_2)$.

\end{proof}

\subsection{Case: $\alpha<\beta<|a|_p$.}

\begin{rk} In this case the open ball with radius ${{\alpha\beta}\over{|a|_p}}$ and center $x_1$ is the maximal
Siegel disk for fixed point $x_1$. The fixed point $x_2$ is attractive and its basin of attraction is the set $$A(x_2)=\C_p\setminus(V_{{\alpha\beta}\over{|a|_p}}(0)\cup\mathcal P).$$
\end{rk}

\begin{lemma}\label{l8} If $\alpha<\beta<|a|_p$, then the dynamical system generated by $\zeta_{\alpha,\beta}(r)$ has the following properties:
\begin{itemize}
\item[1.] ${\rm Fix}(\zeta_{\alpha,\beta})=\{r: 0\leq r<{{\alpha\beta}\over{|a|_p}}\}\cup\{{{\alpha\beta}\over{|a|_p}}:\, if \,
\grave a={{\alpha\beta}\over{|a|_p}}\}\cup\{{|a|_p}\}$.
\item[2.] If $r={{\alpha\beta}\over{|a|_p}}$ and $\grave a<{{\alpha\beta}\over{|a|_p}}$, then
$\zeta^n_{\alpha,\beta}(r)=\grave a$ for any $n\geq 1$.
\item[3.] If $r>{{\alpha\beta}\over{|a|_p}}$, then
$$\lim_{n\to\infty}\zeta^n_{\alpha,\beta}(r)={|a|_p}.$$
\end{itemize}
\end{lemma}
\begin{proof} Parts 1 and 2 are straightforward.

3. By definition of
$\zeta_{\alpha,\beta}(r)$, for $r>\beta$ we have
$\zeta_{\alpha,\beta}(r)=|a|_p$, i.e., the function is constant.
For $r=\beta$ we have $\zeta_{\alpha,\beta}(\beta)=\grave\beta\geq |a|_p$ and by condition
$|a|_p>\beta$, we get
$\zeta_{\alpha,\beta}(\beta)>\beta$. Consequently,
$$\lim_{n \to \infty}\zeta_{\alpha,\beta}^n(\beta)=|a|_p.$$

Let $r\in (\alpha,\beta)$ and $m=\max\left\{n| \, n\in N\cup\{0\}, \, n<\log_{\beta\over |a|_p}{\alpha\over\beta}\right\}$. So we get $$(\alpha,\beta)=\bigcup_{k=0}^m\left[{{\beta^{k+2}\over{|a|_p^{k+1}}}},{{\beta^{k+1}\over{|a|_p^k}}}\right)\cup\left(\alpha,{{\beta^{m+2}\over{|a|_p^{m+1}}}}\right).$$
Denote $A_k=\left[{{\beta^{k+2}\over{|a|_p^{k+1}}}},{{\beta^{k+1}\over{|a|_p^k}}}\right)$ and $A=\left(\alpha,{{\beta^{m+2}\over{|a|_p^{m+1}}}}\right)$. We note that if $\alpha<r<\beta$, then $\zeta_{\alpha,\beta}(r)={{|a|_pr}\over\beta}$. From this we have the following $$f(A)=A_m,$$ $$f(A_k)=A_{k-1}, \ \ 1\leq k\leq m$$
$$f(A_0)=[\beta,|a|_p),$$ i.e., if $r\in(\alpha,\beta)$, then $f^m(r)\geq\beta$ and from above, we have $$\lim_{n \to \infty}\zeta_{\alpha,\beta}^n(r)=|a|_p.$$

For $r=\alpha$ we have $\zeta_{\alpha,\beta}(\alpha)=\grave\alpha\geq {{\alpha|a|_p}\over\beta}>\alpha$. Consequently,
$$\lim_{n \to \infty}\zeta_{\alpha,\beta}^n(\alpha)=|a|_p.$$

Assume now  ${{\alpha\beta}\over |a|_p}<r<\alpha$ then
$\zeta_{\alpha,\beta}(r)={{|a|_p r^2}\over {\alpha\beta}}$,
$\zeta'_{\alpha,\beta}(r)={{2|a|_p r}\over {\alpha\beta}}>2$. Therefore, there exists $n_0\in N$ such that
for $n\geq n_0$ we get $\zeta_{\alpha,\beta}^n(r)>\alpha$ and consequently
$$\lim_{n \to \infty}\zeta_{\alpha,\beta}^n(r)=|a|_p.$$
\end{proof}

By Lemma \ref{lf2} and Lemma
\ref{l8} we get

\begin{thm}\label{t8} If $\alpha<\beta<|a|_p$, then
 the $p$-adic dynamical system generated by function (\ref{fd}) has the following properties:
\begin{itemize}
\item[1.]
\begin{itemize}
\item[1.1)] $SI(x_1)=U_{{{\alpha\beta}\over |a|_p}}(0)$.
\item[1.2)] $x_2\in S_{|a|_p}(0)$.
\end{itemize}
\item[2.] If $x\in S_{{\alpha\beta}\over{|a|_p}}(0)$, then one of the
following two possibilities holds:
\begin{itemize}
\item[2.1)] There exists $k\in N$ and
$\mu_k<{{\alpha\beta}\over {|a|_p}}$ such that $f^m(x)\in
S_{\mu_k}(0)$ for any $m\geq k$ and $f^m(x)\in
S_{{\alpha\beta}\over{|a|_p}}(0)$ if $m\leq k-1$.
\item[2.2)] The trajectory $\{f^k(x), k\geq 1\}$ is a subset of
$S_{{\alpha\beta}\over{|a|_p}}(0)$.
\end{itemize}
\item[3.] The fixed point $x_2$ is attractive and $$A(x_2)=C_p\setminus(V_{{{\alpha\beta}\over |a|_p}}(0)\cup\mathcal P).$$
\end{itemize}
\end{thm}
\begin{proof}
1. By Lemma \ref{lf2} and part 1 of Lemma \ref{l8} we see that
 spheres $S_r(0)$, $r<{{\alpha\beta}\over{|a|_p}}$ and $S_{|a|_p}(0)$ are invariant for $f$.
 Thus $SI(x_1)=U_{{\alpha\beta}\over{|a|_p}}(0)$.

Note that $|a|_p>\beta$ and $|d|_p=\beta$. From this $|x_2|_p=|a-d|_p=|a|_p$, i.e. $x_2\in S_{|a|_p}(0)$.

2. If $x\in S_{{\alpha\beta}\over{|a|_p}}(0)$ then by (\ref{f2}) we have
$$
|f(x)|_p={|ax+b|_p\over{|a|_p}}\leq{{\alpha\beta}\over{|a|_p}}.
$$
If $|f(x)|_p<{{\alpha\beta}\over{|a|_p}}$ then there is
$\mu_1<{{\alpha\beta}\over{|a|_p}}$ such that $f^m(x)\in
S_{\mu_1}(0)$ for any $m\geq 1$ (see part 1 of Lemma \ref{l8}). So in this case $k=1$.

If $|f(x)|_p={{\alpha\beta}\over{|a|_p}}$ then we consider the
following
$$
|f^2(x)|_p={|af(x)+b|_p\over{{|a|_p}}}\leq{{\alpha\beta}\over{|a|_p}}.
$$
Now, if $|f^2(x)|_p<{{\alpha\beta}\over{|a|_p}}$ then there is
$\mu_2<{{\alpha\beta}\over{|a|_p}}$ such that $f^m(x)\in
S_{\mu_2}(0)$ for any $m\geq 2$. So in this case $k=2$.

If $|f^2(x)|_p={{\alpha\beta}\over{|a|_p}}$ then we can continue
the argument and get the following inequality
$$|f^k(x)|_p\leq{{\alpha\beta}\over{|a|_p}}.$$
Hence in each step we may have two possibilities:
$|f^k(x)|_p={{\alpha\beta}\over{|a|_p}}$ or
$|f^k(x)|_p<{{\alpha\beta}\over{|a|_p}}$. In case
$|f^k(x)|_p<{{\alpha\beta}\over{|a|_p}}$ there exists $\mu_k$ such
that $f^m(x)\in S_{\mu_k}(0)$  for any $m\geq k$. If
$|f^k(x)|_p={{\alpha\beta}\over{|a|_p}}$ for any $k\in \N$ then
$\{f^k(x), k\geq 1\}\subset S_{{\alpha\beta}\over{|a|_p}}(0)$.

3. In this case $x_2$ will be attractive fixed point, i.e. $$|f'(x_2)|_p={{|b+dx_2|_p}\over{|b+ax_2|_p}}<1,$$
because we have $|b+dx_2|_p=\beta|a|_p$ and
$|b+ax_2|_p=|a|^2_p$.

From Lemma \ref{lf2} and part 3 of Lemma
\ref{l8} we have $$\lim_{n\to\infty}f^n(x)\in
S_{|a|_p}(0)$$ for all $x\in S_r(0)\setminus\mathcal P$, $r>{{\alpha\beta}\over{|a|_p}}$.

Let $x\in S_{|a|_p}(0)$. By simple calculation we get
$$
|f(x)-x_2|_p=|x-x_2|_p\cdot{{|dx+b|_p}\over{|(x-\hat x_1)|_p|(x-\hat x_2)|_p}}.
$$
By $|dx+b|_p=\beta|a|_p$ and $|x-\hat x_1|_p=|x-\hat x_2|_p=|a|_p$,
we have that $|f(x)-x_2|_p<|x-x_2|_p$ for any $x\in S_{|a|_p}(0)\setminus\mathcal P$.
Consequently $$\lim_{n\to\infty}f^n(x)=x_2,  \ \ \mbox{for all} \ \ x\in S_r(0)\setminus\mathcal P, \ \ r>{{\alpha\beta}\over{|a|_p}},$$
i.e. $A(x_2)=\mathbb C_p\setminus(V_{{\alpha\beta}\over{|a|_p}}(0)\cup\mathcal P)$.
\end{proof}

\section{Periodic orbits.}

 In this section we are interested in periodic orbits of
 the dynamical system and we find an invariant set (with respect to function (\ref{fd})), which contains all periodic orbits.

Define the following sets
$$I_1=\{r: \, 0<r<\alpha\} \ \ {\rm if} \ \ |a|_p<\beta;$$
$$I_2=\{r: \, 0<r<\alpha \ \ \mbox{or} \ \ \alpha<r<\beta\} \ \ {\rm if} \ \
|a|_p=\beta;$$ $$I_3=\left\{r: \, 0<r<{{\alpha\beta}\over{|a|_p}}\right\} \ \ {\rm if} \ \ |a|_p>\beta;$$
 and we denote $I=I_1\cup I_2\cup I_3$.

From previous sections we have the following
\begin{cor} The sphere $S_r(0)$ is invariant for $f$ if and only if $r\in I$.
\end{cor}

\begin{thm}{\label{ab}}
For every closed ball $V_{\r}(c)\subset S_r(0), \, r\in I$ the
following equality holds $$f(V_{\r}(c))=V_{\r}(f(c)).$$
\end{thm}
\begin{proof}
From inclusion $V_{\r}(c)\subset S_r(0)$ we have $|c|_p=r$.
Let $x\in V_{\r}(c)$, i.e. $|x-c|_p\leq\r$, then
\begin{equation}{\label{ab1}}
|f(x)-f(c)|_p=|x-c|_p\cdot\frac{|adcx+ab(x+c)-bcx+b^2|_p}{|(x-\hat
x_1)(x-\hat x_2)(c-\hat x_1)(c-\hat x_2)|_p}.
\end{equation}
We have $|x|_p=r$, because $x\in V_{\r}(c)\subset S_r(0)$. Consequently,
$$|adcx|_p=|ad|_pr^2, \ \ |b^2|_p=\alpha^2\beta^2, \ \ |ab(x+c)|_p\leq\alpha\beta|a|_p r, \ \ |bcx|_p=\alpha\beta r^2.$$

If $|a|_p<\beta$, then $r\in I_1$, i.e., $0<r<\alpha$. So we have $$\max\{|ad|_pr^2, \,
\alpha\beta|a|_p r, \, \alpha\beta r^2, \,
\alpha^2\beta^2\}=\alpha^2\beta^2$$ and $$|(x-\hat x_1)(x-\hat
x_2)(c-\hat x_1)(c-\hat x_2)|_p=\alpha^2\beta^2.$$
Using this equality by (\ref{ab1}) we get $|f(x)-f(c)|_p=|x-c|_p\leq\r.$

If $|a|_p=\beta$, then $r\in I_2$, i.e.,
$r\in(0,\alpha)$ or $r\in(\alpha,\beta)$.
For $0<r<\alpha$ we have
$$|f(x)-f(c)|_p=|x-c|_p\leq\r.$$
Let $\alpha<r<\beta$, so we have $|d|_p=\beta$. Then $$\max\{|ad|_pr^2, \,
\alpha\beta|a|_p r, \, \alpha\beta r^2, \,
\alpha^2\beta^2\}=\max\{\beta^2r^2, \,
\alpha\beta^2 r, \, \alpha\beta r^2, \,
\alpha^2\beta^2\}=\beta^2r^2$$ and $$|(x-\hat x_1)(x-\hat
x_2)(c-\hat x_1)(c-\hat x_2)|_p=r^2\beta^2.$$
Consequently $|f(x)-f(c)|_p=|x-c|_p\leq\r.$

If $|a|_p>\beta$, then $r\in I_3$, i.e., $0<r<{{\alpha\beta}\over{|a|_p}}<\alpha$. So we have $$\max\{|ad|_pr^2, \,
\alpha\beta|a|_p r, \, \alpha\beta r^2, \,
\alpha^2\beta^2\}=\alpha^2\beta^2 \ \ \mbox{and} \ \ |(x-\hat x_1)(x-\hat
x_2)(c-\hat x_1)(c-\hat x_2)|_p=\alpha^2\beta^2.$$
Using these equalities by (\ref{ab1}) we get $|f(x)-f(c)|_p=|x-c|_p\leq\r.$
This completes the
proof.
\end{proof}

Let $X\subset \C_p$ be a compact open subset of $\C_p$.  Consider the dynamical system $(X,f)$, where
  $f : X \to X$ is a given function.

\begin{defn}\label{d1}{\cite{DS}} Let $X$ be a compact open subset of $\C_p$ and let $f : X \to X$ be a
rational function. Assume, in addition, that $f'(x)$ has no roots in $X$. Then $f$ is uniformly
locally scaling with local scalar $C(a) = |f'(a)|_p$, i.e., there exist $r > 0$ such that for any $a\in X$,
$|f(x)-f(y)|_p = |f'(a)|_p |x-y|_p$ whenever $x, y \in V_r(a)$.
\end{defn}

We can define the notions (uniformly) locally isometric as follows.

\begin{defn}\label{d2}{\cite{DS}} A map $f : X \to X$ is uniformly locally isometric if there exists a constant $r > 0$
such that for all $a \in X$, $|f(x)-f(y)|_p=|x-y|_p$ whenever $x, y \in V_r(a)$.
\end{defn}

By a similar argument as in Definition \ref{d1}, we have the following criteria.
\begin{pro}\label{p1}{\cite{DS}} Let $X$ be a compact open subset of $\C_p$ and let $f : X \to X$ be a rational function.
Then $f$ is uniformly locally isometric if and only if $|f'(a)|_p = 1$ for all $a \in X$.
\end{pro}

For function (\ref{fd}) we have the following

\begin{lemma}\label{t11}
Let $r\in I$ and suppose $x\in S_r(0)$ is arbitrary. Then
$|f'(x)|_p=1.$
\end{lemma}

\begin{proof}
We have
$$|f'(x)|_p=\frac{|adx^2-bx^2+2abx+b^2|_p}{|x^4+2dx^3+d^2x^2+2bx^2+2bdx+b^2|_p}.$$

Note that $|x|_p=r$, $|b|_p=\alpha\beta$ and $|d|_p\leq\beta$.
Now, if $|a|_p<\beta$, then $r\in I_1$ i.e., $0<r<\alpha\leq\beta$ and we have that
$|f'(x)|_p=\dfrac{\alpha^2\beta^2}{\alpha^2\beta^2}=1.$

Let $|a|_p=\beta$. Then $r\in I_2$, i.e., $0<r<\alpha$ or $\alpha<r<\beta$. If $0<r<\alpha$, then
$|f'(x)|_p=\dfrac{\alpha^2\beta^2}{\alpha^2\beta^2}=1.$
If $\alpha<r<\beta$, then $|d|_p=\beta$ and we have
$|f'(x)|_p=\dfrac{\beta^2r^2}{\beta^2r^2}=1.$

If $|a|_p>\beta$, then $r\in I_3$, i.e., $0<r<{{\alpha\beta}\over{|a|_p}}<\alpha$ and
$|f'(x)|_p=\dfrac{\alpha^2\beta^2}{\alpha^2\beta^2}=1.$
Thus, $|f'(x)|_p=1$ for all $x\in S_r(0)$.
\end{proof}

\begin{cor}\label{c2}
If $r\in I$ then $f:S_r(0)\rightarrow S_r(0)$ is uniformly locally isometric.
\end{cor}

\begin{lemma}\label{ab2}
If $c\in{S}_r(0)$ for $r\in{I}$, then
$$|f(c)-c|_p=\left\{\begin{array}{lllllllllllllll}
\frac{r^2}{\alpha}, \ \ \ \ \ \ \ \ {\rm if} \ \ r\in I_1, \, \alpha<\beta\\[2mm]
\frac{r^2}{\alpha}, \ \ \ \ \ \ \ \ {\rm if} \ \ r\in I_1, \, |a-d|_p=\alpha=\beta\\[2mm]
\frac{r^2|a-d|_p}{\alpha^2}, \ \ {\rm if} \ \ r\in I_1, \, |a-d|_p<\alpha=\beta, \, r<|a-d|_p\\[2mm]
\frac{r^3}{\alpha^2}, \ \ \ \ \ \ \ \ {\rm if} \ \ r\in I_1, \, |a-d|_p<\alpha=\beta, \, r>|a-d|_p\\[2mm]
<\frac{r^3}{\alpha^2}, \ \ \ \ \ {\rm if} \ \ r\in I_1, \, |a-d|_p<\alpha=\beta, \, |c-x_2|_p<r=|a-d|_p\\[2mm]
\frac{r^3}{\alpha^2}, \ \ \ \ \ \ \ \ {\rm if} \ \ r\in I_1, \, |a-d|_p<\alpha=\beta, \, |c-x_2|_p=r=|a-d|_p\\[2mm]
\frac{r^2|a-d|_p}{\alpha\beta}, \ \ {\rm if} \ \ r\in I_2, \, r<\alpha, \, r<|a-d|_p\\[2mm]
\frac{r^3}{\alpha\beta}, \ \ \ \ \ \ \ \ {\rm if} \ \ r\in I_2, \, r<\alpha, \, r>|a-d|_p\\[2mm]
<\frac{r^3}{\alpha\beta}, \ \ \ \ \ {\rm if} \ \ r\in I_2, \, r<\alpha, \, r=|a-d|_p, \, |c-x_2|_p<r\\[2mm]
\frac{r^3}{\alpha\beta}, \ \ \ \ \ \ \ \ {\rm if} \ \ r\in I_2, \, r<\alpha, \, r=|a-d|_p, \, |c-x_2|_p=r\\[2mm]
\frac{r|a-d|_p}{\beta}, \ \ \ \, {\rm if} \ \ r\in I_2, \, r>\alpha, \, r<|a-d|_p\\[2mm]
\frac{r^2}{\beta}, \ \ \ \ \ \ \ \ \ {\rm if} \ \ r\in I_2, \, r>\alpha, \, r>|a-d|_p\\[2mm]
<\frac{r^2}{\beta}, \ \ \ \ \ \ {\rm if} \ \ r\in I_2, \, r>\alpha, \, r=|a-d|_p, \, |c-x_2|_p<r\\[2mm]
\frac{r^2}{\beta}, \ \ \ \ \ \ \ \ \ {\rm if} \ \ r\in I_2, \, r>\alpha, \, r=|a-d|_p, \, |c-x_2|_p=r\\[2mm]
\frac{|a|_pr^2}{\alpha\beta}, \ \ \ \ \ \, {\rm if} \ \ r\in I_3.
\end{array}
\right.$$
\end{lemma}

\begin{proof}
By simple calculation we get the following
$$|f(c)-c|_p=\dfrac{|c|_p^2|a-d-c|_p}{|c-\hat x_1|_p|c-\hat x_2|_p}.$$
Using this equality the strong triangle inequality of $p$-adic norm one complies  the proof.
\end{proof}

\begin{rk}\label{r11}
If $r\neq|a-d|_p$, then by Lemma \ref{ab2} we have that $|f(c)-c|_p$ depends on $r$, but
does not depend on $c\in S_r(0)$ itself, therefore we define $\rho_r=|f(c)-c|_p$, if
$c\in S_r(0)$, $r\neq|a-d|_p$.
\end{rk}

\begin{thm}\label{ca}
If $c\in{S}_r(0)$ for $r\in{I}$, $r\neq|a-d|_p$ then\\
\begin{itemize}
\item[1.] $|f^{n+1}(c)-f^n(c)|_p=\rho_r$ for any $n\geq{1}$.\\
\item[2.] $f(V_{\rho_r}(c))=V_{\rho_r}(c).$\\
\item[3.] If $V_{\theta}(c)\subset{S}_r(0)$ is an invariant for $f$, then $\theta\geq{\rho_r}.$
\end{itemize}
\end{thm}
\begin{proof}
The proof is similar to the proof of Theorem 11 in \cite{RS2}.
\end{proof}

\begin{thm}\label{t9}
Let $r\in I$ and $r\neq|a-d|_p$. If dynamical system $(S_r(0), f)$ has $k$-periodic orbit $y_0\to y_1\to ...\to y_k\to y_0$,
then the following statements hold:
\begin{itemize}
\item[1.] $y_i\in V_{\rho_r}(y_0)$ for all $i\in\{1, 2, ..., k\}$;
\item[2.] $|(f^k(x))'_{x=y_i}|_p=1$, for all $i\in\{0, 1, ..., k\}$, i.e., character of periodic points is indifferent;
\item[3.] if $\rho\leq\rho_r$, then we have $f(S_{\rho}(y_i)\setminus\mathcal P)\subset S_{\rho}(y_{i+1})$ for any $i\in\{0,1,...k-1\}$ and\\
$f(S_{\rho}(y_k)\setminus\mathcal P)\subset S_{\rho}(y_0).$
\end{itemize}
\end{thm}

\begin{proof}
1. Follows from the part 2 of Theorem \ref{ca}, i.e., closed ball with radius $\rho_r$ is the minimal invariant ball for function $f$.

2. By Lemma \ref{t1} we have $|f'(x)|_p=1$, for arbitrary $x\in S_r(0)$.

3. This result is follows from the proof of Theorem \ref{ab}, i.e., for all $x\in V_{\rho}(c)\subset S_r(0)$ we have $|f(x)-f(c)|_p=|x-c|_p$.
\end{proof}

\section{Ergodicity properties of the dynamical system in $\mathbb{Q}_p$.}

In this section we consider the dynamical system of function (\ref{fd}) in
$\Q_p$ and we are
interested to study ergodicity properties of the dynamical system on each invariant sphere.

Consider the dynamical system $(X,T,\mu)$, where $T:X\rightarrow
X$ is a measure preserving transformation, and $\mu$ is a
probability measure. The dynamical system is called {\it
ergodic} if for every
invariant set $V$ we have $\mu(V)=0$ or $\mu(V)=1$ (see
\cite{Wal}).

We assume that the square root $\sqrt{d^2-4b}$ exists in $\Q_p$.

For each $r\in I$, $r\neq|a-d|_p$ consider a measurable space $(S_r(0),\mathcal
B)$, here $\mathcal B$ is the algebra generated by closed
subsets of $S_r(0)$. Every element of $\mathcal B$ is a union of
some balls $V_{\rho}(c)$.

A measure $\bar\mu:\mathcal B\rightarrow \mathbb{R}$ is said to be
\emph{Haar measure} if it is defined by $\bar\mu(V_{\rho}(c))=\rho$.

Note that $S_r(a)=V_r(a)\setminus V_{r\over p}(a)$. So, we have
$\bar\mu(S_r(0))=r(1-{1\over p})$.

We consider normalized (probability) Haar measure:
$$\mu(V_{\rho}(c))={{\bar\mu(V_{\rho}(c))}\over{\bar\mu(S_r(0))}}={{p\rho}\over{(p-1)r}}.$$

By Theorem \ref{ab} we conclude that $f$ preserves the measure
$\mu$, i.e.
$
\mu(f(V_{\rho}(c)))=\mu(V_{\rho}(c)).
$

\begin{thm}\label{t6}
Let $r\in I$ and $r\neq|a-d|_p$. If $p\neq2$, then the dynamical system $(S_r(0), f, \mu)$ is
not ergodic.
\end{thm}

\begin{proof}
Let $r\in I$ and $r\neq|a-d|_p$. By the part 2 of Theorem \ref{ca}, the ball
$V_{\rho_r}(c)$ is invariant for any $c\in S_r(0)$. Using Lemma
\ref{ab2} we get
$$\mu(V_{\rho_r}(c))={{p\rho_r}\over
(p-1)r}=\left\{\begin{array}{lllllllll}
\frac{pr}{\alpha(p-1)}, \ \ \ \ \ \ \ \ {\rm if} \ \ r\in I_1, \, \alpha<\beta\\[2mm]
\frac{pr}{\alpha(p-1)}, \ \ \ \ \ \ \ \ {\rm if} \ \ r\in I_1, \, |a-d|_p=\alpha=\beta\\[2mm]
\frac{pr|a-d|_p}{\alpha^2(p-1)}, \ \ \ \ \ \ \, {\rm if} \ \ r\in I_1, \, |a-d|_p<\alpha=\beta, \, r<|a-d|_p\\[2mm]
\frac{pr^2}{\alpha^2(p-1)}, \ \ \ \ \ \ \  {\rm if} \ \ r\in I_1, \, |a-d|_p<\alpha=\beta, \, r>|a-d|_p\\[2mm]
\frac{pr|a-d|_p}{\alpha\beta(p-1)}, \ \ \ \ \ \ \, {\rm if} \ \ r\in I_2, \, r<\alpha, \, r<|a-d|_p\\[2mm]
\frac{pr^2}{\alpha\beta(p-1)}, \ \ \ \ \ \ \ {\rm if} \ \ r\in I_2, \, r<\alpha, \, r>|a-d|_p\\[2mm]
\frac{p|a-d|_p}{\beta(p-1)}, \ \ \ \ \ \ \ \ {\rm if} \ \ r\in I_2, \, r>\alpha, \, r<|a-d|_p\\[2mm]
\frac{pr}{\beta(p-1)}, \ \ \ \ \ \ \ \ \, {\rm if} \ \ r\in I_2, \, r>\alpha, \, r>|a-d|_p\\[2mm]
\frac{pr|a|_p}{\alpha\beta(p-1)}, \ \ \ \ \ \ \ {\rm if} \ \ r\in I_3.
\end{array}
\right.$$

If $r\in I_1$, then $0<r<\alpha\leq\beta$. Since $r$ radius is a value of a $p$-adic norm
we have ${{pr}\over\alpha}\leq 1$. Thus $0<\mu(V_{\rho_r}(c))\leq{1\over{p-1}}$.

Let $r\in I_2$, i.e., $0<r<\alpha$ or $\alpha<r<\beta$. We note that $|a-d|_p\leq\beta$.
If $r<\alpha$, then ${{pr}\over\alpha}\leq 1$ and $0<\mu(V_{\rho_r}(c))\leq{1\over{p-1}}$.

If $\alpha<r<\beta$, then $pr\leq\beta$. In case $r<|a-d|_p$ we assume that $|a-d|_p=\beta$. Consequently,
 $\mu(V_{\rho_r}(c))={p\over{p-1}}>1$. We note that $\mu$ is a probability measure, i.e.,
  $\mu(V_{\rho}(c))\leq 1$ for every closed ball $V_{\rho}(c)\in \mathcal B$.
Therefore, we have $|a-d|_p<\beta$ for case $\alpha<r<\beta$ and $r<|a-d|_p$. So $p|a-d|_p\leq\beta$.
Thus, $0<\mu(V_{\rho_r}(c))\leq{1\over{p-1}}$.

If $r\in I_3$, then we have $0<r<{{\alpha\beta}\over{|a|_p}}$
and ${{pr|a|_p}\over{\alpha\beta}}\leq 1$.
Thus $0<\mu(V_{\rho_r}(c))\leq{1\over{p-1}}$.

Therefore if $p\neq 2$, i.e., $p\geq 3$, then the dynamical system
$(S_r(0), f, \mu)$ is not ergodic for all $r\in I$, $r\neq|a-d|_p$.
Theorem is proved.
\end{proof}

 Recall  $\mathbb Z_2=\{x\in \Q_2: \, |x|_2\leq 1\}$. So we have $1+2\mathbb Z_2=S_1(0)$.

 The following theorem gives a criterion of ergodicity for the rational functions mapping $S_1(0)$ to itself:

\begin{thm}{\cite{M}}{\label{crit}}
Let $f,g: 1+2\mathbb Z_2\rightarrow 1+2\mathbb Z_2$ be
polynomials whose coefficients are $2$-adic integers.

Set $f(x)=\sum_ia_ix^i$, $g(x)=\sum_ib_ix^i$, and
$$A_1=\sum_{i \, {\rm odd}}a_i, \ \ A_2=\sum_{i \, {\rm even}}a_i, \ \ B_1=\sum_{i \, {\rm odd}}b_i, \ \
B_2=\sum_{i \, {\rm even}}b_i.$$

The rational function $R={f\over g}$ is ergodic if and only if one
of the following situations occurs:

(1) $A_1=1({\rm mod}\,4)$, $A_2=2({\rm mod} \, 4)$, $B_1=0({\rm mod} \, 4)$ and
$B_2=1({\rm mod} \, 4)$.

(2) $A_1=3({\rm mod} 4)$, $A_2=2({\rm mod} \, 4)$, $B_1=0({\rm mod} \, 4)$ and
$B_2=3({\rm mod} \, 4)$.

(3) $A_1=1({\rm mod} \, 4)$, $A_2=0({\rm mod} \, 4)$, $B_1=2({\rm mod} \, 4)$ and
$B_2=1({\rm mod} \, 4)$.

(4) $A_1=3({\rm mod} \, 4)$, $A_2=0({\rm mod} \, 4)$, $B_1=2({\rm mod} \, 4)$ and
$B_2=3({\rm mod} \, 4)$.

(5) One of the previous cases with $f$ and $g$ interchanged.
\end{thm}

But, in this paper we will study ergodicity of the dynamical system $(S_r(0), f, \mu)$ for any $r\in I$.
For this purpose we can not use Theorem \ref{crit} directly, because the radius of the sphere is arbitrary.

Let $r=p^l$ and a function $f: S_{p^l}(0)\rightarrow S_{p^l}(0)$ be given.
Denote $(f\circ g)(t)=f(g(t))$.

Consider $x=g(t)=p^{-l}t, \, t=g^{-1}(x)=p^lx$ then
it is easy to see that\\ $f\circ g: S_1(0)\rightarrow
S_{p^l}(0)$.
Consequently, $g^{-1}\circ f\circ g: S_1(0)\rightarrow
S_1(0)$.

Let $\mathcal B$ (resp. $\mathcal B_1$) be the algebra generated by closed
subsets of $S_{p^l}(0)$ (resp. $S_1(0)$), and $\mu$ (resp. $\mu_1$)
be normalized Haar measure on $\mathcal B$ (resp. $\mathcal B_1$).

\begin{thm}{\cite{RS2}}{\label{erg1}} The dynamical system
$(S_{p^l}(0), \, f, \, \mu)$ is ergodic if and only if\\
$(S_1(0), \, g^{-1}\circ f\circ g, \, \mu_1)$ is ergodic.
\end{thm}

Now using the above mentioned results for
$f(x)={{ax^2+bx}\over{x^2+dx+b}}$, when $p=2$ and\\
$f: S_r(0)\rightarrow
S_r(0)$ we prove the following theorem.

\begin{thm}{\label{erg2}}
Let $p=2$ and $r\in I$. Then the dynamical system $(S_r(0), f, \mu)$ is ergodic if and only if one
of the following conditions occurs:

(1) $|a|_2<\beta$, $|d|_2=\beta$ and $r={\alpha\over 2}$.

(2) $|a|_2=\beta$ and $r={\beta\over 2}$.

(3) $|a|_2>\beta$ and $r={{\alpha\beta}\over {2|a|_2}}$.
\end{thm}

\begin{proof} Let $r=2^l$, $\alpha=2^m$, $\beta=2^k$, $|a|_2=2^s$ and $|d|_2=2^q$.
By Remark \ref{r2} it suffices to prove theorem for the case $\alpha\leq\beta$, i.e., $m\leq k$.
 Since
 $\alpha=|\hat x_1|_2$, $\beta=|\hat x_2|_2$, $d=-\hat x_1-\hat x_2$ and $b=\hat x_1 \hat x_2$,
 we have $q\leq k$ and $|b|_2=2^{m+k}$.

Note that $$I=\left\{\begin{array}{lll}
I_1, \ \ \ {\rm if} \ \ |a|_2<\beta\\[2mm]
I_2, \ \ \ {\rm if} \ \ |a|_2=\beta\\[2mm]
I_3, \ \ \ {\rm if} \ \ |a|_2>\beta.
\end{array}
\right.$$

Let $|a|_2<\beta$, i.e., $s<k$. Then the sphere $S_{2^l}(0)$ is invariant for $f$ if and only if $l<m$.
In $f:S_{2^l}(0)\rightarrow S_{2^l}(0)$ we change $x$ by
$x=g(t)=2^{-l}t$. Note that $x\in S_{2^l}(0)$, consequently
$|x|_2=2^l|t|_2=2^l$, $|t|_2=1$ and the function
$g^{-1}(f(g(t))):S_1(0)\rightarrow S_1(0)$ has the following form
\begin{equation}{\label{k}}
g^{-1}(f(g(t)))={{{{2^{-l}a}\over b}t^2+t}\over{{{2^{-2l}}\over b}t^2+{{2^{-l}d}\over b}t+1}}.
\end{equation}

For the numerator of (\ref{k}) we have $$\left|{{2^{-l}a}\over b}t^2\right|_2=2^{l+s-(m+k)}\leq 2^{-2}, \ \ \ \
\left|{{2^{-2l}}\over b}t^2\right|_2=2^{2l-(m+k)}\leq 2^{-2}$$ $$ \mbox{and} \ \
\left|{{2^{-l}d}\over b}t\right|_2=2^{l+q-(m+k)}\leq 2^{-1}.$$

Consequently,
$${{2^{-l}a}\over b}t^2+t=:\gamma_1(t), \ \ \mbox{is such that} \ \ \gamma_1: 1+2 \mathbb Z_2\rightarrow 1+2 \mathbb Z_2$$
and
$${{2^{-2l}}\over b}t^2+{{2^{-l}d}\over b}t+1=:\gamma_2(t)    \ \ \mbox{is such that} \ \  \gamma_2: 1+2 \mathbb Z_2\rightarrow 1+2 \mathbb Z_2.$$

Hence the function (\ref{k}) satisfies all conditions of Theorem
\ref{crit}, therefore using this theorem we get
 $$A_1=1, \ \
A_2={{2^{-l}a}\over b}, \ \ B_1={{2^{-l}d}\over b} \ \ \mbox{and} \ \ B_2=1+{{2^{-2l}}\over b}.$$

Moreover,
$$A_1=1({\rm mod}\,4), \ \ A_2=0({\rm mod}\,4), \ \
B_1\in 2^{m+k-(l+q)}(1+2 \mathbb Z_2) \ \ \mbox{and} \ \ B_2=1({\rm
mod}\,4).$$

By this relations and Theorem \ref{crit} we get $m+k-(l+q)=(m-l)+(k-q)=1$. Note that $l<m$ and $q\leq k$.
Therefore we conclude that the dynamical system $(S_1(0), \,
g^{-1}\circ f\circ g, \, \mu_1)$ is ergodic if and only if  $q=k$ and $l=m-1$, i.e. $|d|_2=\beta$ and $r={\alpha\over 2}$. Consequently, by Theorem \ref{erg1}, $(S_r(0), f, \mu)$ is ergodic iff $|d|_2=\beta$
and $r={\alpha\over 2}$.

The proofs for the cases $|a|_2=\beta$ and $|a|_2>\beta$ are similar.
\end{proof}

\end{document}